\documentclass[11pt,a4paper]{amsart}
\usepackage{graphicx}
\usepackage{amsmath,amsfonts}
\usepackage{amsthm,amssymb,latexsym}
\usepackage[active]{srcltx}
\usepackage[usenames]{color}
\usepackage[utf8]{inputenc} 
\DeclareMathOperator{\rank}{rank}

\newcommand{\bfs}{\boldsymbol}

\newtheorem{theorem}{Theorem}[section]
\newtheorem{corollary}[theorem]{Corollary}

\newtheorem*{fact*}{Fact}
\newtheorem{proposition}[theorem]{Proposition}
\theoremstyle{definition}

\newtheorem*{claim}{Claim}
\newtheorem{remark}[theorem]{Remark}
\numberwithin{equation}{section}

\newcommand{\Z}{\mathbb Z}

\newcommand{\A}{\mathbb A}
\newcommand{\F}{\mathbb F}
\newcommand{\K}{\mathbb K}

\newcommand{\Pp}{\mathbb P}

\newcommand{\fp}{\F_{\hskip-0.7mm p}}
\newcommand{\fq}{\F_{\hskip-0.7mm q}}

\newcommand{\cfq}{\overline{\F}_{\hskip-0.7mm q}}
\def\ifm#1#2{\relax \ifmmode#1\else#2\fi}

\newcommand{\klk}    {\ifm {,\ldots,} {$,\ldots,$}}

\begin{document}

\title[Systems of diagonal equations over finite fields]{On the number of simultaneous solutions of certain diagonal equations over finite fields}
\author[M. P\'erez]{
Mariana P\'erez${}^{1,3}$}
\author[M. Privitelli]{
Melina Privitelli${}^{1,2}$}

\address{${}^{1}$ Consejo Nacional de Investigaciones Científicas y Técnicas (CONICET),
Ar\-gentina}
\address{${}^{2}$
Universidad Nacional de Gene\-ral Sarmiento, Instituto de Ciencias, J.M. Guti\'errez 1150
(B1613GSX) Los Polvorines, Buenos Aires, Argentina}
\email{mprivite@campus.ungs.edu.ar}
\address{${}^{3}$
Universidad Nacional de Hurlingham, Instituto de Tecnología e Ingeniería, Av. Gdor. Vergara 2222
(B1688GEZ) Villa Tesei, Buenos Aires, Argentina}
\email{mariana.perez@unahur.edu.ar}

\thanks{The authors were partially supported by the grants
PIP CONICET 11220130100598, PIO CONICET-UNGS 14420140100027, ICI-UNGS 30/1146 and PICTO-UNAHUR-2019-00012}.

%\date{\today}%

\begin{abstract}
In this paper we obtain explicit estimates and existence results on the number of $\fq$-rational solutions of certain systems defined by families of diagonal equations over finite fields. Our approach relies on the study of the geometric properties of the varieties defined by the systems involved. We apply these results  to a generalization of Waring's problem and the distribution of solutions of congruences modulo a prime number.
\end{abstract}

\maketitle

\section{Introduction}
Let $\fq$ be the finite field of $q$ elements. It is a classical  problem  to determine or estimate the number $N$ of $\fq$--rational solutions (i.e. solutions with coordinates in $\fq$) of systems of polynomial equations over $\fq$ (see, e.g., \cite{MuPa13}). Particulary, the systems of diagonal equations
\begin{equation}\label{eq: system intro}\left \{\begin{array}{ccl}
a_{11}X_1^{d_1} & + a_{12}X_2^{d_2} + \cdots + & a_{1t}X_t^{d_t} = b_1 \\
a_{21}X_1^{d_1} & + a_{22}X_2^{d_2} + \cdots + & a_{2t}X_t^{d_t} = b_2 \\
\;\vdots &  &\quad\vdots \\
a_{n1}X_1^{d_1} &+ a_{n2}X_2^{d_2} + \cdots +& a_{nt}X_t^{d_t} = b_n,
\end{array}\right.
\end{equation}
with $b_1,\ldots,b_n\in \fq$, have been considered in the literature because the study of its set of $\fq$--rational solutions has several applications to different areas of mathematics, such as the theory of cyclotomy, Waring’s problem and the graph coloring problem (see, e.g. \cite{LiNi83} and \cite{MuPa13}). Additionally, information on the number $N$ is very useful  in different  aspects of coding theory such as 
the weight distribution of some cyclic codes (\cite{Zh10} and \cite{Zh15}) and the
covering radius of certain cyclic codes (\cite{He85} and \cite{MoCa03}). 

The case $n=1$ of the system \eqref{eq: system intro} has been extensively studied. In general, there are no explicit formulas of  the number $N$, except for some very particular diagonal equations. For this reason, many articles focus on providing estimates on the number $N$ (see, e.g. \cite{LiNi83}, \cite{MuPa13} and \cite{Weil49}). In \cite{PP20}, we obtain existence results and estimates on the number of $\fq$-rational solutions of some variants of diagonal equations.
%: Carlitz's equation, Dickson's equation and Markoff-Hurwitz-type's equations. 

In comparison with diagonal equations, there are much fewer results about the number of $\fq$--rational solutions of systems of the type \eqref{eq: system intro}. There are explicit formulas for the number $N$ for some very particular cases (see, e.g., \cite{Cao2016} and  \cite{W98}). On the other hand,  A. Tiet\"{a}v\"{a}inen  provides existence results for some special families of systems of type \eqref{eq: system intro} (see \cite{T64}, \cite{T65 A}, \cite{T65} and \cite{T67}). In   \cite{Spackman79} and \cite{Spackman81}  K. Spackman  obtains  the following estimate
  using elementary methods involving character sums which holds under certain conditions on a parameter  which measures the extent to which the coefficients' matrix is non--singular over $\fq$: 
$$N=q^{t-n}+\mathcal{O}(q^{(t-1)/2}),$$
where the implied constant depends only on $d_1, \ldots, d_t, n$ and $t$, but it is not explicitly given. 

In this article we obtain an explicit estimate on the number $N$ using tools of  algebraic geometry. %Our condition on the coefficients' matrix is less restrictive than Spackman's. 
Furthemore, our technique allows us  to consider the case $d_1=\cdots=d_t$ to study more general systems of diagonal equations such us systems of Markoff-Hurwitz-type equations, systems of Dickson equations, and systems of deformed diagonal equations. 
More precisely, we consider the following more general system: \begin{equation}\label{eq: system intro 2}\left \{\begin{array}{ccl}
a_{11}X_1^{d_1} & + a_{12}X_2^{d_2} + \cdots + & a_{1t}X_t^{d_t} = g_1(X_1,\ldots,X_k)\\
a_{21}X_1^{d_1} & + a_{22}X_2^{d_2} + \cdots + & a_{2t}X_t^{d_t} =g_2(X_1,\ldots,X_k) \\
\;\vdots &  &\quad\vdots \\
a_{n1}X_1^{d_1} &+ a_{n2}X_2^{d_2} + \cdots +& a_{nt}X_t^{d_t} = g_n(X_1,\ldots,X_k),
\end{array}\right.
\end{equation} where $g_1, \ldots,g_n \in \fq[X_1, \ldots,X_k]$ are such that $g_j\in \fq$ for $1 \leq j \leq n$ or $0\leq \deg(g_j)<d_t$ for $1 \leq j \leq n$ and there exists $1\leq i\leq n$ such that $0<\deg(g_i)$. The main result of this article is the following.
\begin{theorem} \label{estimate: number of solutions}
Suppose that all the $(n\times n)$--submatrix of the coefficients' matrix has rank $n$, $d_1>\cdots>d_t \geq 2$, and $\mathrm{char}(\fq)$ does  not divide $d_i$ for $1\leq i \leq t$.
 We have the following estimates on $N$: 
\begin{itemize}
	\item  If $g_i \in \fq$ and $n\leq t-2$, then $N$ satisfies:
	$$\big|N-q^{t-n}\big|\leq  q^{\frac{t-n+1}{2}}(6n\cdot d_1)^{t+1}.$$
 \item If $0\leq\deg(g_j)<d_t$, there exists $1\leq i \leq n$ such that $\deg(g_i)>0$, and $1 \leq n \leq \frac{t-1}{2}$, satisfies:
$$\big|N-q^{t-n}\big|\leq  q^{\frac{t-n+k}{2}}(6n\cdot d_1)^{t+1}.$$
\end{itemize}
\end{theorem}
Our techniques also allow us to replace $X_i^{d_i}$ by $h_i(X_i)$ for $1\leq i \leq n$ where $ h_i\in \fq[T]$ and $\deg (h_i)=d_i$. In particular, we examine the case where $h_i(X_i)$ is the Dickson's polynomial over $\fq$ $D_{d_i}(X_i,a)$ of degree $d_i$ with parameter $a \in \fq$ and we obtain a similar result to Theorem \ref{estimate: number of solutions} for this case.

The paper is organized as follows. In Section 2 we collect the notions of algebraic geometry we use throughout the article. In Section 3 we 
study the geometric properties of the varieties associated to the system \eqref{eq: system intro 2} and we settle Theorem \ref{estimate: number of solutions}.  As a consequence, we obtain existence results of $\fq$--rational solutions of these type of  systems. We also study a particular example when $g_i=b_i X_1^{c_{i_1}}\cdots X_n^{c_{i_n}}-a_i$ (the generalized Markoff-Hurwitz-type equations systems). In Section 4 we consider some variants of systems of diagonal equations, such as the Dickson equations.
Finally, in Section 5 we study two applications of our estimates: Generalized Waring's problem over finite fields and the distributions of solutions of systems of congruences module a prime number.
\section{Basic notions of algebraic geometry}
%\label{sec: notions of algebraic geometry}
%
In this section we collect the basic definitions and facts of
algebraic geometry that we need in the sequel. We use standard
notions and notations which can be found in, e.g., \cite{Kunz85},
\cite{Shafarevich94}.

Let $\K$ be any of the fields $\fq$ or $\cfq$, the closure of $\fq$. We denote by $\A^r$
the affine $r$--dimensional space $\cfq{\!}^{r}$ and by $\Pp^r$ the
projective $r$--dimensional space over $\cfq{\!}^{r+1}$. Both spaces
are endowed with their respective Zariski topologies over $\K$, for
which a closed set is the zero locus of a set of polynomials of
$\K[X_1,\ldots, X_r]$, or of a set of homogeneous polynomials of
$\K[X_0,\ldots, X_r]$.

A subset $V\subset \Pp^r$ is a {\em projective variety defined over}
$\K$ (or a projective $\K$--variety for short) if it is the set of
common zeros in $\Pp^r$ of homogeneous polynomials $F_1,\ldots, F_m
\in\K[X_0 ,\ldots, X_r]$. Correspondingly, an {\em affine variety of
	$\A^r$ defined over} $\K$ (or an affine $\K$--variety) is the set of
common zeros in $\A^r$ of polynomials $F_1,\ldots, F_{m} \in
\K[X_1,\ldots, X_r]$. We think a projective or affine $\K$--variety
to be equipped with the induced Zariski topology. We shall denote by
$\{F_1=0,\ldots, F_m=0\}$ or $V(F_1,\ldots,F_m)$ the affine or
projective $\K$--variety consisting of the common zeros of
$F_1,\ldots, F_m$.

In the remaining part of this section, unless otherwise stated, all
results referring to varieties in general should be understood as
valid for both projective and affine varieties.

A $\K$--variety $V$ is {\em irreducible} if it cannot be expressed
as a finite union of proper $\K$--subvarieties of $V$. Further, $V$
is {\em absolutely irreducible} if it is $\cfq$--irreducible as a
$\cfq$--variety. Any $\K$--variety $V$ can be expressed as an
irredundant union $V=\mathcal{C}_1\cup \cdots\cup\mathcal{C}_s$ of
irreducible (absolutely irreducible) $\K$--varieties, unique up to
reordering, called the {\em irreducible} ({\em absolutely
	irreducible}) $\K$--{\em components} of $V$.

For a $\K$--variety $V$ contained in $\Pp^r$ or $\A^r$, its {\em
	defining ideal} $I(V)$ is the set of polynomials of $\K[X_0,\ldots,
X_r]$, or of $\K[X_1,\ldots, X_r]$, vanishing on $V$. The {\em
	coordinate ring} $\K[V]$ of $V$ is the quotient ring
$\K[X_0,\ldots,X_r]/I(V)$ or $\K[X_1,\ldots,X_r]/I(V)$. The {\em
	dimension} $\dim V$ of $V$ is the length $n$ of a longest chain
$V_0\varsubsetneq V_1 \varsubsetneq\cdots \varsubsetneq V_n$ of
nonempty irreducible $\K$--varieties contained in $V$. %A
%$\K$--variety $V$ is called {\em equidimensional} if all the
%irreducible $\K$--components of $V$ are of the same dimension. In
%such a case,
We say that $V$ has {\em pure dimension} $n$ if every irreducible
$\K$--component of $V$ has dimension $n$. A $\K$--variety of $\Pp^r$
or $\A^r$ of pure dimension $r-1$ is called a $\K$--{\em
	hypersurface}. A $\K$--hypersurface of $\Pp^r$ (or $\A^r$) can also
be described as the set of zeros of a single nonzero polynomial of
$\K[X_0,\ldots, X_r]$ (or of $\K[X_1,\ldots, X_r]$).

The {\em degree} $\deg V$ of an irreducible $\K$--variety $V$ is the
maximum of $|V\cap L|$, considering all the linear spaces $L$ of
codimension $\dim V$ such that $|V\cap L|<\infty$. More generally,
following \cite{Heintz83} (see also \cite{Fulton84}), if
$V=\mathcal{C}_1\cup\cdots\cup \mathcal{C}_s$ is the decomposition
of $V$ into irreducible $\K$--components, we define the degree of
$V$ as
$$\deg V:=\sum_{i=1}^s\deg \mathcal{C}_i.$$
The degree of a $\K$--hypersurface $V$ is the degree of a polynomial
of minimal degree defining $V$. %Another property is that the degree
%of a dense open subset of a $\K$--variety $V$ is equal to the degree
%of $V$.
%
We shall use the following {\em B\'ezout inequality} (see
\cite{Heintz83}, \cite{Fulton84}, \cite{Vogel84}): if $V$ and $W$
are $\K$--varieties of the same ambient space, then
\begin{equation}\label{eq: Bezout}
\deg (V\cap W)\le \deg V \cdot \deg W.
\end{equation}

Let $V\subset\A^r$ be a $\K$--variety, $I(V)\subset \K[X_1,\ldots,
X_r]$ its defining ideal and $x$ a point of $V$. The {\em dimension}
$\dim_xV$ {\em of} $V$ {\em at} $x$ is the maximum of the dimensions
of the irreducible $\K$--components of $V$ containing $x$. If
$I(V)=(F_1,\ldots, F_m)$, the {\em tangent space} $\mathcal{T}_xV$
to $V$ at $x$ is the kernel of the Jacobian matrix $(\partial
F_i/\partial X_j)_{1\le i\le m,1\le j\le r}(x)$ of $F_1,\ldots, F_m$
with respect to $X_1,\ldots, X_r$ at $x$. We have
$\dim\mathcal{T}_xV\ge \dim_xV$ (see, e.g., \cite[page
94]{Shafarevich94}). The point $x$ is {\em regular} if
$\dim\mathcal{T}_xV=\dim_xV$; otherwise, $x$ is called {\em
	singular}. The set of singular points of $V$ is the {\em singular
	locus} of $V$; it is a closed $\K$--subvariety of
$V$. A variety is called {\em nonsingular} if its singular locus is
empty. For projective varieties, the concepts of tangent space,
regular and singular point can be defined by considering an affine
neighborhood of the point under consideration.

%Let $V$ and $W$ be irreducible affine $\K$--varieties of the same
%dimension and $f:V\to W$ a regular map with $\overline{f(V)}=W$,
%where $\overline{f(V)}$ denotes the closure of $f(V)$ with respect
%to the Zariski topology of $W$. Such a map is called {\em dominant}.
%Then $f$ induces a ring extension $\K[W]\hookrightarrow \K[V]$ by
%composition with $f$. We say that the dominant map $f$ is {\em
%	finite} if this extension is integral, namely each element
%$\eta\in\K[V]$ satisfies a monic equation with coefficients in
%$\K[W]$. A dominant finite
%morphism is necessarily closed. Another fact %concerning
%we shall use is that the preimage $f^{-1}(S)$ of an irreducible
%closed subset $S\subset W$ under a dominant finite morphism $f$ is
%of pure dimension $\dim S$ (see, e.g., \cite[\S 4.2,
%Proposition]{Danilov94}).
%
%----------------------------------------------------------------
%----------------------------------------------------------------
%----------------------------------------------------------------
%----------------------------------------------------------------
%
\subsection{Rational points}
Let $\Pp^r(\fq)$ be the $r$--dimensional projective space over $\fq$
and $\A^r(\fq)$ the $r$--dimensional $\fq$--vector space $\fq^r$.
For a projective variety $V\subset\Pp^r$ or an affine variety
$V\subset\A^r$, we denote by $V(\fq)$ the set of $\fq$--rational
points of $V$, namely $V(\fq):=V\cap \Pp^r(\fq)$ in the projective
case and $V(\fq):=V\cap \A^r(\fq)$ in the affine case. For an affine
variety $V$ of dimension $n$ and degree $\delta$, we have the
following bound (see, e.g., \cite[Lemma 2.1]{CaMa06}):
\begin{equation}\label{eq: upper bound -- affine gral}
|V(\fq)|\leq \delta\, q^n.
\end{equation}
On the other hand, if $V$ is a projective variety of dimension $n$
and degree $\delta$, then we have the following bound (see
\cite[Proposition 12.1]{GhLa02a} or \cite[Proposition 3.1]{CaMa07};
see \cite{LaRo15} for more precise upper bounds):
\begin{equation*}\label{eq: upper bound -- projective gral}
|V(\fq)|\leq \delta\, p_n,
\end{equation*}
where $p_n:=q^n+q^{n-1}+\cdots+q+1=|\Pp^n(\fq)|$.
%
%----------------------------------------------------------------
%----------------------------------------------------------------
%----------------------------------------------------------------
%----------------------------------------------------------------
%
\subsection{Complete intersections}\label{subsec: complete intersections}
Elements $F_1,\ldots, F_m$ in $\mathbb{K}[X_1,\ldots,X_r]$ or
$\mathbb{K}[X_0,\ldots,X_r]$ form a \emph{regular sequence} if $F_1$
is nonzero and no $F_i$ is zero or a zero divisor in the quotient
ring $\mathbb{K}[X_1,\ldots,X_r]/ (F_1,\ldots,F_{i-1})$ or
$\mathbb{K}[X_0,\ldots,X_r]/ (F_1,\ldots,F_{i-1})$ for $2\leq i \leq
m$.  In such a case, the (affine or projective) variety $V := V (F_1,\ldots,F_m)$ they define is equidimensional of dimension $r-m$, and is called a \emph{set–theoretic complete intersection}.
 Furthermore, $V$ is called an (ideal–theoretic) \emph{complete intersection} if its ideal $I(V)$ over $K$ can be generated by $r-m$ polynomials.    
%For  given  positive  integers $a_1, \ldots,a_r$, we define the weight $\wt({\boldsymbol{X^\alpha}})$ of a monomial $\boldsymbol{X^{\alpha}}:=X_1^{\alpha_1}\cdots X_r^{\alpha_r}$ as $\wt({\boldsymbol{X^\alpha}}):=\sum_{i=1}^r a_i\cdot \alpha_i$. The weight $\wt(f)$ of an arbitrary element $f\in  \mathbb{K}[X_1,\ldots,X_r]$ is the highest weight of all the monomials with nonzero coefficients arising in the dense representation of $f$.  
%\begin{lemma} \cite[Lemma 5.4]{MaPePr19} \label{lemma: regular sequences 2}
%Let $F_1 \klk F_m \in \K[X_1 \klk X_{r}]$. For an assignment of
%posi\-tive integer weights $\wt$ to the variables $X_1 \klk X_{r}$,
%denote by $F_1^{\wt} \klk F_m^{\wt}$ the components of highest
%weight of $F_1 \klk F_m$. If $F_1^{\wt} \klk F_m^{\wt}$ form a
%regular sequence in $\K[X_1 \klk X_{r}]$, then $F_1 \klk F_m$
%form a regular sequence in $\K[X_1 \klk X_{r}]$.
%\end{lemma}
%
 We shall frequently use the following criterion to
prove that a variety is a complete intersection (see, e.g.,
\cite[Theorem 18.15]{Eisenbud95}).
\begin{theorem}\label{theorem: eisenbud 18.15}
	Let $F_1,\ldots,F_m\in\mathbb{K}[X_1,\ldots,X_r]$ be polynomials
	which form a regular sequence and let
	$V:=V(F_1,\ldots,F_m)\subset\A^r$. Denote by
	$(\partial\bfs{F}/\partial \bfs{X})$ the Jacobian matrix of
	$F_1,\ldots,F_m$ with respect to $X_1,\ldots,X_r$. If the subvariety
	of $V$ defined by the set of common zeros of the maximal minors of
	$(\partial\bfs{F}/\partial \bfs{X})$ has codimension at least one in
	$V$, then $F_1,\ldots,F_m$ define a radical ideal. In particular,
	$V$ is a complete intersection.
\end{theorem}

If $V\subset\Pp^r$ is a complete intersection defined over $\K$ of
dimension $r-m$, and $F_1 ,\ldots, F_m$ is a system of homogeneous
generators of $I(V)$, the degrees $d_1,\ldots, d_m$ depend only on
$V$ and not on the system of generators. Arranging the $d_i$ in such
a way that $d_1\geq d_2 \geq \cdots \geq d_m$, we call $(d_1,\ldots,
d_m)$ the {\em multidegree} of
$V$. In this case, a stronger version of %the B\'ezout inequality
\eqref{eq: Bezout} holds, called the {\em B\'ezout theorem} (see,
e.g., \cite[Theorem 18.3]{Harris92}):
\begin{equation}\label{eq: Bezout theorem}
\deg V=d_1\cdots d_m.
\end{equation}
A complete intersection $V$ is called {\em normal} if it is {\em
	regular in codimension 1}, that is, the singular locus
$\mathrm{Sing}(V)$ of $V$ has codimension at least $2$ in $V$,
namely $\dim V-\dim \mathrm{Sing}(V)\ge 2$ (actually, normality is a
general notion that agrees on complete intersections with the one we
define here). A fundamental result for projective complete
intersections is the Hartshorne connectedness theorem (see, e.g.,
\cite[Theorem VI.4.2]{Kunz85}): if $V\subset\Pp^r$ is a complete
intersection defined over $\K$ and $W\subset V$ is any
$\K$--subvariety of codimension at least 2, then $V\setminus W$ is
connected in the Zariski topology of $\Pp^r$ over $\K$. Applying the
Hartshorne connectedness theorem with $W:=\mathrm{Sing}(V)$, one
deduces the following result.
\begin{theorem}\label{theorem: normal complete int implies irred}
	If $V\subset\Pp^r$ is a normal complete intersection, then $V$ is
	absolutely irreducible.
\end{theorem}
%
%------------------------------------------------------------------
%------------------------------------------------------------------
%------------------------------------------------------------------
%------------------------------------------------------------------
%------------------------------------------------------------------
%------------------------------------------------------------------
%------------------------------------------------------------------
%------------------------------------------------------------------
%
%
%
%
%
%
%
% 
\section{Systems of diagonal equations} 
Let  $t,n, d_1,\ldots,d_t, k$ be positive integers such that $n \leq \frac{t-1}{2}$ ,  $1 \leq k \leq n$, $d_1>\cdots>d_t \geq 2$, and $\mathrm{char}(\fq)$ does  not divide $d_i$ for $1\leq i \leq t$. Let $X_1, \ldots, X_t$ be indeterminates over $\fq$ and let $g_1, \ldots,g_n \in \fq[X_1, \ldots,X_k]$ such that $g_j\in \fq$ for $1 \leq j \leq n$ or $0\leq \deg(g_j)<d_t$ for $1 \leq j \leq n$ and there exists $1\leq i\leq n$ such that $0<\deg(g_i)$. 

We consider the following system of $n$ deformed diagonal equations with $t$ unknowns
\begin{equation}\label{eq: system}\left \{\begin{array}{ccl}
a_{11}X_1^{d_1} & + a_{12}X_2^{d_2} + \cdots + & a_{1t}X_t^{d_t} = g_1(X_1, \ldots,X_k) \\
a_{21}X_1^{d_1} & + a_{22}X_2^{d_2} + \cdots + & a_{2t}X_t^{d_t} = g_2(X_1, \ldots,X_k) \\
\;\vdots &  &\quad\vdots \\
a_{n1}X_1^{d_1} &+ a_{n2}X_2^{d_2} + \cdots +& a_{nt}X_t^{d_t} = g_n(X_1, \ldots,X_k).
\end{array}\right.
\end{equation}
Let $A=[a_{ij}]\in \fq^{n \times t}$ be the coefficients' matrix of the above system.  Assume that $A$ satisfies the following hypothesis:
\smallskip

$(H)$ All $(n\times n)$--submatrix of $A$ has $\rank=n$.

% Indeed, if $\rank(A)=k$ for $1 \leq k \leq n-1$, we can obtain a system equivalent to \eqref{eq} with $k$ equations and $t$ unknowns.
Let $N$ denote the number of $\fq$--rational solutions of \eqref{eq: system}. The purpose of this paper is to give an estimate on the number $N$. 
 To do this, we consider the following polynomials $f_j\in \fq[X_1, \ldots,X_t]$
 \begin{equation*}\label{def: f'i}
 f_j:=a_{j1}X_1^{d_1}+a_{j2}X_2^{d_2}+\cdots+a_{jt}X_t^{d_t}-g_j(X_1,\ldots,X_k),\,\,1 \leq j \leq n.
 \end{equation*}
Without loss of generality we can assume that $\deg(g_j)>0$ for all $1\leq j\leq n$ or $g_j\in \fq$ for all $1\leq j\leq n$ .  Let $V:=V(f_1, \ldots,f_n)\subset \A^t$ be the $\fq$--affine variety defined by $f_1, \ldots,f_n$. We shall study some facts concerning the geometry of $V$. 
 %First suppose that $g_j=b_j \in \fq$, $1\leq j\leq n$. 
 
 Since $\rank(A)=n$, there exists an invertible matrix $M\in \fq^{n \times n}$  such that $M \cdot A=\hat{A}$ is in row echelon form, that is
\begin{equation} \label{eq:Aes}
\hat{A}:=
\left(
\begin{array}{ccccccc}
a_{1i_1} & \cdots &  & \cdots &  &\cdots &  a_{1t}
	\\
0 & a_{2i_2}  &&   &  &\cdots & a_{2t}\\
	\vdots & & \ddots  &   & \vdots &  & \vdots
	\\
0& \cdots &  0  & \cdots & a_{ni_n}& \cdots   & a_{nt}
	\end{array}
	\right),
	\end{equation}
with $1=i_1 < i_2 < \ldots<i_n\leq t$ are the indices of the corresponding  pivots.
 Let  $\hat{V}\subset \A^t$ the following $\fq$-- affine variety  
 $$\hat{V}:=\Bigg\{(x_1, \ldots,x_t)\in \A^t: \hat{A}\cdot \begin{pmatrix} x_1^{d_1} \\ \vdots\\x_t^{d_t}\end{pmatrix} = \begin{pmatrix} \hat{g}_1 \\ \vdots \\ \hat{g}_n\end{pmatrix}\Bigg\},$$ 
 where $ \begin{pmatrix}\hat{g}_1 \\ \vdots \\ \hat{g}_n\end{pmatrix}= M\cdot \begin{pmatrix} g_1 \\ \vdots \\ g_n\end{pmatrix}$,
namely $\hat{V}=V(\hat{f}_1, \ldots, \hat{f}_n)\subset \A^t$ is the  $\fq$--affine variety defined by $\hat{f}_j:=a_{ji_j}X_{i_j}^{d_{i_j}}+ \cdots +a_{jt}X_t^{d_{t}}-\hat{g}_j$, for $ 1 \leq j \leq n$.  It is clear that $V=\hat{V}^.$
\begin{claim}\label{claim: f sucesion regular}
 $\hat{f}_1, \ldots, \hat{f}_n$ form a regular sequence of $\fq[X_1, \cdots,X_t]$. Indeed, consider the graded lexicographic order  of $\fq[X_1, \cdots,X_t]$ with $X_1> \cdots > X_t$. With this order we have that $Lt(\hat{f}_j)=a_{ji_j}X_{i_j}^{d_{i_j}}$, where $Lt(\hat{f}_j)$ denotes the leading term of the polynomial $\hat{f}_j$. Thus, $Lt(\hat{f}_j)$ are relatively prime and then they form a Gröbner basis of the ideal $J$ generated by $\hat{f_j}$, $1\leq j \leq n$ (see, e.g., \cite[\S 2.9, Proposition 4]{CLO92}). Hence,
the initial of the ideal $J$ is generated by $Lt(\hat{f}_1),\ldots,Lt(\hat{f}_n)$, which form a regular sequence
of $\fq[X_1,\ldots,X_t]$. Therefore, by \cite[Proposition 15.15]{Eisenbud95}, the polynomials $\hat{f}_{1},\ldots,\hat{f}_{n}$ form
a regular sequence of $\fq[X_1,\ldots,X_t]$. We conclude that $V(\hat{f}_{1},\ldots,\hat{f}_{n})$ is a set complete intersection of $\A^t$ of pure dimension $t-n$.
\end{claim} 
Thus, we obtain the following result.
\begin{theorem}\label{teo: dimension of V}
 $V$ is a set-theoretic complete intersection of pure dimension $t-n$.
\end{theorem}

Let $C$ be the following set of $\A^t$:
\begin{equation}\label{def C} C:=\bigg\{{\bf{x}}\in V:\,\,\, \rank \bigg(\frac{\partial f}{\partial {\bf{X}}}\bigg)({\bf{x}}) <n\bigg\},\end{equation}
where the $(n\times t)$--matrix $\frac{\partial f}{\partial \bf{X}}$ is the jacobian matrix of the polynomials $f_j$, $1\leq j \leq n.$
Suppose that $g_j \in \fq$ or $  0<\deg(g_j) < d_t$, $1\leq j\leq n$ and $1 \leq n \leq \frac{t-1}{2}$. Assume that $A$, the coefficients' matrix of the system \eqref{eq: system}, satisfies the hypothesis $(H)$.
Observe that
\begin{equation*} \label{eq:matriz jacobiana factorizada case gral}
\frac{\partial f}{\partial \bf{X}}:= \left(
\begin{array}{c|c}
M_1 & M_2
\end{array}
\right)
\end{equation*}
where $M_1$ is a $(n \times k)$--matrix defined by
$$M_1:=\left(
\begin{array}{ccccccc}
a_{11}d_1X_1^{d_1-1}+\frac{\partial g_1}{\partial X_1} & \cdots &  & \cdots &   a_{1k}d_kX_k^{d_k-1}+ \frac{\partial g_1}{\partial X_k}
\\
\vdots &  \vdots  &   & \vdots &   \vdots
\\
a_{n1}d_1X_1^{d_1-1}+\frac{\partial g_n}{\partial X_1}& \cdots &    & \cdots &   a_{nk}d_kX_k^{d_k-1}+\frac{\partial g_n}{\partial X_k}
\end{array} \right)$$
and $M_2$ is a  $n \times(t-k)$--matrix defined by
$$M_2:=\left(
\begin{array}{ccccccc}
a_{1k+1}d_{k+1}X_{k+1}^{d_{k+1}-1}& \cdots &  & \cdots &   a_{1t}d_tX_t^{d_t-1}
\\
\vdots &  \vdots  &   & \vdots &   \vdots
\\
a_{n{k+1}}d_{k+1}X_{k+1}^{d_{k+1}-1}& \cdots &    & \cdots &   a_{nt}d_tX_t^{d_t-1}
\end{array} \right)$$

%Let $C$ be the set defined in \eqref{def C}.
\begin{proposition} \label{prop: singular locus g no constante}  The dimension of $C$ is at most $k-1$ if $\deg(g_j)>0$ for $1\leq j \leq n$ and this dimension is at most $0$ if $g_i\in \fq$ for $1\leq i \leq n$. In particular, the dimension of the singular locus of $V$ is at most $k-1$ or $0$ respectively.
\end{proposition}
\begin{proof}
Let ${\bf{x}}\in C$. We claim that ${\bf{x}}$ has at least $t-k-n+1$ coordinates equal to zero among $x_{k+1}$ and $x_t$. Indeed, if ${\bf{x}}$ has at most $t-k-n$ coordinates equal to zero among $x_{k+1}$ and $x_t$ then ${\bf{x}}$ has at least $n$ nonzero coordinates. Suppose that these coordinates are $x_{k+1},\ldots,x_{k+n}$. Then, we consider the following $(n\times n)$--submatrix of $M_2$:
$$M_{2,n}({\bf{x}})=\left(
\begin{array}{ccccccc}
a_{1k+1}d_{k+1}x_{k+1}^{d_{k+1}-1}& \cdots &  & \cdots &   a_{1{k+n}}d_{k+n}x_{k+n}^{d_{k+n}-1}
\\
\vdots &  \vdots  &   & \vdots &   \vdots
\\
a_{n{k+1}}d_{k+1}x_{k+1}^{d_{k+1}-1}& \cdots &    & \cdots &   a_{nk+n}d_{k+n}x_{k+n}^{d_{k+n}-1}
\end{array} \right).$$ We have that $M_{2,n}$ can be written as follows:
\begin{equation}\label{eq: M2n factorizada}
M_{2,n} ({\bf{x}})=
\left(
\begin{array}{ccccccc}
a_{1k+1}d_{k+1} & \cdots &  & \cdots &   a_{1k+n}d_{k+n}
	\\
	\vdots &  \vdots  &   & \vdots &   \vdots
	\\
a_{nk+1}d_{k+1}& \cdots &    & \cdots &   a_{nk+n}d_{k+n}
	\end{array}
	\right)\cdot	\left(
\begin{array}{ccccc}
x_{k+1}^{d_{k+1}-1} &  &    \cdots & 0 
	\\
%	&  \ddots & 
%	\\	
\vdots & & \ddots&  \vdots  
\\	
0 &  & \cdots &       x_{k+n}^{d_{k+n}-1}
	\end{array}
	\right)
	\end{equation}

 From $(H)$ and the fact of $d_i\neq 0$ for all $1\leq i \leq t$, the determinant of 
 $$\left(
\begin{array}{ccccccc}
a_{1k+1}d_{k+1} & \cdots &  & \cdots &   a_{1k+n}d_{k+n}
	\\
	\vdots &  \vdots  &   & \vdots &   \vdots
	\\
a_{nk+1}d_{k+1}& \cdots &    & \cdots &   a_{nk+n}d_{k+n}
	\end{array}
	\right)$$
is nonzero. On the other hand, since $x_i\neq 0$ for $k+1\leq i \leq k+n$ we have that the determinant of the diagonal matrix of the right side of  \eqref{eq: M2n factorizada} is nonzero. Hence $M_2({\bf{x}})$ has rank $n$ and so $\frac{\partial f}{\partial \bf{X}}$ does.

 In order to estimate the dimension of $C$ it is enough to study the set of points of $C$ which have exactly $t-n-k+1$ coordinates equal to zero. We take a point ${\bf{x}}$ with this characteristic.  Without loss of generality, suppose that these coordinates are $x_{k+1},\ldots,x_{t-n+1}$. Now, we replace $ X_{k+1}=\cdots=X_{t-n+1}=0$ in \eqref{eq: system}  and we obtain a new system of $n$ equations and $k+n-1$ unknowns. From hypothesis $(H)$ and following the arguments of  the claim above, we deduce  that ${\bf{x}}$ belongs to a subvariety of $V$ of dimension $k+n-1-n= k-1$. We conclude that the dimension of $C$ is at most $k-1$.

On the other hand, if $g_j \in \fq$ for $1 \leq j \leq n$,  then with similar arguments, we obtain a system of $n$ equations and $n-1$ unknowns. From hypothesis $(H)$ we conclude that the dimension of the set of solution of the system is at most $0$.
\end{proof}
%Recall that $A\in \fq^{n \times t}$ is the coefficients' matrix of the system defined in \eqref{eq: system}, $V\subset \A^t$ is the variety $V=V(f_1,%\ldots,f_n)$ $f_i$, $1\leq i \leq n$ are defined in \eqref{def: f'i} and $g_i\in \fq[X_1,\ldots,X_k]$, $1\leq i \leq n$.
From Proposition \ref{prop: singular locus g no constante} and Theorem \ref{theorem: normal complete int implies irred}, we have the following result.
 \begin{corollary} \label{coro:radical ideal V} Let $k,n,t$ positive integers
 	% with $1\leq  n$ and 
 	such that $1\leq k \leq n$ and $A$ satisfies the hypothesis $(H)$.  If $g_j\in \fq$ for $1\leq j\leq n$ and $n\leq {t-2}$	or $\deg(g_j)\geq 0$ for $1\leq j \leq n$,  there exists $1 \leq i \leq n$ such that $0 < \deg(g_i)$ and $n\leq \frac{t-1}{2}$, then
 the singular locus of $V$ has codimension at least $2$ in $V$ and  $(f_1,\ldots,f_n)$ is a radical ideal.
 \end{corollary}
% 
% \begin{proof}
% From the above proposition and  \cite[Proposition 15.15]{Eisenbud95}, the relations between $n$, $k$ and $t$ we deduce that $(f_1,\ldots,f_n)$ is a radical ideal.
% \end{proof}
%From Corollary \ref{coro:radical ideal V} we deduce  the following result.
Then, we obtain the following result.
\begin{theorem}\label{teo: V es interseccion completa}
With the same hypotheses as in Corollary \ref{coro:radical ideal V}, $V=V(f_1 \klk f_n) \subset \A^t$ is a complete intersection of degree at most $ d_{i_1} \cdots d_{i_n}$, where $i_1,\ldots,i_n$ are the pivots of the matrix defined in \eqref{eq:Aes}.	
\end{theorem}

\subsection{The geometry of the projective closure}\label{geo proyectiva}
Consider the embedding of $\A^t$ into the projective space $\Pp^t$
which assigns to any $\bfs{x}:=(x_1,\ldots, x_t)\in\A^t$ the point
$(1:x_1:\dots:x_t)\in\Pp^t$. Then the closure
$\mathrm{pcl}(V)\subset\Pp^t$ of the image of $V$ under this
embedding in the Zariski topology of $\Pp^t$ is called the
projective closure of $V$. The points of $\mathrm{pcl}(V)$ lying
in the hyperplane $\{X_0=0\}$ are called the points of
$\mathrm{pcl}(V)$ at infinity.

It is well--known that $\mathrm{pcl} (V)$ is the $\fq$--variety of
$\mathbb{P}^t$ defined by the homogenization
$F^h\in\fq[X_0,\ldots,X_t]$ of each polynomial $F$ belonging to the
ideal $(f_1\klk f_n)\subset\fq[X_1,\ldots,X_t]$ (see, e.g.,
\cite[\S I.5, Exercise 6]{Kunz85}). Denote by $(f_1\klk f_t)^h$
the ideal generated by all the polynomials $F^h$ with $F\in
(f_1\klk f_n)$. Since $(f_1\klk f_n)$ is radical it turns
out that $(f_1\klk f_n)^h$ is also a radical ideal (see, e.g.,
\cite[\S I.5, Exercise 6]{Kunz85}). Furthermore, $\mathrm{pcl}
(V)$ has pure dimension $t-n$ (see, e.g., \cite[Propositions
I.5.17 and II.4.1]{Kunz85}) and degree equal to $\deg V$ (see, e.g.,
\cite[Proposition 1.11]{CaGaHe91}).

Now we discuss the behaviour of $\mathrm{pcl} (V)$ at infinity. Recall that $V=V(\hat{f}_1, \klk, \hat{f}_n) \subset \A^n$, where $\hat{f}_j:=a_{ji_j}X_{i_j}^{d_{i_j}}+ \cdots + a_{jt}X_t^{d_{t}}-\hat{g}_j$, for $ 1 \leq j \leq n$, where $1=i_1 < i_2 < \ldots<i_n\leq t$, $\hat{g}_j \in \fq[X_1,\ldots,X_k]$, $1\leq k\leq n$ and $0 \leq \deg(\hat{g}_j) <d_t$.
Hence, the homogenization of each $\hat{f}_j$ is the following polynomial of $\fq[X_0 \klk X_t]:$ $$\hat{f}_j^{ h}:=a_{ji_j}X_{i_j}^{d_{i_j}}+ X_0 \cdot h_j, \, \,  ( 1 \leq j \leq n)$$
where $h_j \in \fq[X_1 \klk X_t],$ $\deg(h_j)<d_{i_j}$, $ 1 \leq j \leq n$.

In particular, it follows that $\hat{f}_j^{h}(0,X_1 \klk X_t)=a_{ji_j}X_{ij}^{d_{ij}}$ for $1 \leq j \leq n$.

\begin{proposition} \label{prop: dimension of pcl en el infinito}
	$V^{\infty}:=\mathrm{pcl}(V) \cap \{X_0=0\} \subset \Pp^{t-1}$ is  a non-singular linear complete intersection of pure dimension $t-n-1$.
\end{proposition}
\begin{proof}
Recall that the projective variety $\mathrm{pcl}(V)$ has pure
dimension $t-n$. Hence, each irreducible component of
$\mathrm{pcl}(V)\cap \{X_0=0\}$ has dimension at least $t-n-1$.
On the other hand, from the definition of $\hat{f}_j^{h}$, $1\leq j \leq n$, we deduce that $ \mathrm{pcl}(V) \cap \{X_0=0\}\subset V(X_{i_1},\ldots,X_{i_n})$. Since $V(X_{i_1},\ldots,X_{i_n})$ is a nonsingular irreducible variety of $\Pp^{t-1}$ of pure dimension $t-n-1$ we obtain that $\mathrm{pcl}(V) \cap \{X_0=0\}=V(X_{i_1},\ldots,X_{i_n})$ and therefore,  the proposition follows.	
\end{proof} 
\begin{corollary} \label{corolario: puntos sing en el infinito} $\mathrm{pcl}(V)$ has not singular points at infinity.
\begin{proof}
From \cite[Lemma 1.1]{GhLa02a} we have that the set of singular points of $\mathrm{pcl}(V)$ lying in $\{X_0=0\}$ is contained in the set of singular points of the variety $\mathrm{pcl}(V) \cap \{X_0=0\}$. Then, taking into account the above proposition we have that $\mathrm{pcl}(V)$ has not singular points at infinity.
\end{proof} 
\end{corollary}
From  Proposition \ref{prop: singular locus g no constante} and Corollary \ref{corolario: puntos sing en el infinito}, we obtain the following result.
\begin{proposition}\label{prop: singular locus pcl V}
If  $ n \leq t-2$ and  $g_j \in \fq $ for $1\leq j\leq n$  the singular locus of $\mathrm{pcl}(V) \subset \Pp^t$ has dimension at most $0$. On the other hand, let  $1\leq k \leq n$ and $ n \leq \frac{t-1}{2}$.   If $0\leq \deg(g_j)$  and there exists $g_i$  such that $\deg(g_i)>0$ for $1\leq i \leq n$, the singular locus of $\mathrm{pcl}(V)$ has  dimension at most $k-1$.
\end{proposition}
We conclude this section with a statement that summarizes all the
facts we need concerning the geometry of the projective closure
$\mathrm{pcl}(V)$.
\begin{proposition}\label{prop: pcl V interseccion completa}  With the same hypotheses as above,
	$\mathrm{pcl}(V) \subset \Pp^t$ is a complete intersection of dimension $t-n$ and degree $d_{i_1}\cdots d_{i_n}$.
\end{proposition}
\begin{proof}
Observe that the following inclusions hold:
\begin{align*}
V(\hat{f_1}^h,\ldots,\hat{f_n}^h) \cap &\{X_0\neq 0\} \subset V(\hat{f_1},\ldots,\hat{f_n}),\\
V(\hat{f_1}^h,\ldots,\hat{f_n}^h) \cap &\{X_0=0\}\subset V(X_{i_1},\ldots,X_{i_n}).
\end{align*}
From Theorem \ref{teo: dimension of V} we have that $V(\hat{f_1},\ldots,\hat{f_n})\subset \A^t$ has dimension $t-n$. It is easy to see that $V(X_{i_1},\ldots,X_{i_n})\subset \Pp^{t}$ has pure dimension $t-n-1$; hence the dimension of $V(\hat{f_1}^h,\ldots,\hat{f_n}^h)\subset \Pp^t$ is at most $t-n$. On the other hand, since $\mathrm{pcl}(V)\subset V(\hat{f_1}^h,\ldots,\hat{f_n}^h)$ is $(t-n)$-dimensional we conclude that $V(\hat{f_1}^h,\ldots,\hat{f_n}^h)$ has dimension $t-n$.

 From Proposition \ref{prop: singular locus g no constante} and taking into account the variety $V(X_{i_1} \klk X_{i_n})$ is non--singular we have  that the codimension of the singular locus of $V(\hat{f_1}^h,\ldots,\hat{f_n}^h)$ is at least $2$. On the other hand $(\hat{f_1}^h,\ldots,\hat{f_n}^h)$ is a radical ideal since  $(\hat{f_1},\ldots,\hat{f_n})$ is one. We conclude that  $V(\hat{f_1}^h,\ldots,\hat{f_n}^h)$ is a normal complete intersection. Hence, from Theorem \ref{theorem: normal complete int implies irred} $V(\hat{f_1}^h,\ldots,\hat{f_n}^h)$ is absolutely irreducible and thus $\mathrm{pcl}(V)=V(\hat{f_1}^h,\ldots,\hat{f_n}^h)$. Finally, from \eqref{eq: Bezout theorem} $\mathrm{pcl}(V)$ has degree $d_{i_1}\cdots d_{i_n}.$
\end{proof}
%----------------------------------------------------------------------------------------------------------------------------------------------------------
%----------------------------------------------------------------------------------------------------------------------------------------------------------
%----------------------------------------------------------------------------------------------------------------------------------------------------------

\subsection{Estimates on the number of $\fq$--rational solutions of  systems of diagonal equations}
Let  $t,n, d_1,\ldots,d_t, k$ be positive integers such that $1 \leq n \leq \frac{t-1}{2}$ ,  $1 \leq k \leq n$ and $d_1>\cdots>d_t \geq 2$. Let $X_1, \ldots, X_t$ be indeterminates over $\fq$ and let $g_1, \ldots,g_n \in \fq[X_1, \ldots,X_k]$ such that  $0\leq \deg(g_j)<d_t$ for $1 \leq j \leq n$. 

In what follows, we shall use an estimate on the number of
$\fq$--rational points of a projective complete intersection due to S. Ghorpade
and G. Lachaud (\cite{GhLa02a}; see also \cite{GhLa02}). In \cite[Theorem 6.1]{GhLa02a}, the authors prove that,
for an irreducible $\fq$--complete intersection $V\subset \Pp^n$
of dimension $r$, multidegree ${\bfs{d}}=(d_1,\ldots,d_{n-r})$ and  singular locus of dimension at most $0\leq s\leq r-1$, the number $|V(\fq)|$ of $\fq$--rational points of $V$ satisfies the estimate:
\begin{equation}\label{eq: estimacion de Ghorpade Lachaud}
\big||V(\fq)|-p_r\big|\leq b_{r-s-1}'(n-s-1,{\bfs{d}})\, \, q^{\frac{r+s+1}{2}}+C_s(V)q^{\frac{r+s}{2}},
\end{equation}
where $p_r:=q^r+q^{r-1}+\cdots+1$,  $b_{r-s-1}'(n-s-1,{\bfs{d}})$ is the  $(r-s-1)$--th primitive Betti of a nonsingular complete intersection in $\Pp^n$
of dimension $r-s-1$ and  multidegree ${\bfs{d}}$,
and $C_s(V):=\sum_{i=m-1}^{m-1+s}b_{i,\ell}(V)+\varepsilon_i$, where
$b_{i,\ell}(V)$ denotes the $i$--th $\ell$--adic Betti number of $V$ for a prime $\ell$ different from $p:=\mathrm{char}(\fq)$ and
$\varepsilon_i:=1$ for even $i$ and $\varepsilon_i:=0$ for odd $i$.
From \cite[Proposition 4.2]{GhLa02a}
\begin{equation}\label{cota de Betti}
b_{r-s-1}'(n-s-1,{\bfs{d}})\leq \binom{n-s}{r-s-1}\cdot (d+1)^{n-s-1},
\end{equation}
where $d:=\max\{d_1,\ldots,d_{n-r}\}$.
On the other hand, from \cite[Theorem 6.1]{GhLa02a}, we have that
\begin{equation*}
C_s(V)\leq 9\cdot 2^{n-r}\cdot ((n-r)d+3)^{n+1}.
\end{equation*}
Denote by  $\mathrm{pcl}(V)(\fq)$ the set $\fq$--rational points of $\mathrm{pcl}(V)$.   From Propositions \ref{prop: singular locus pcl V}  and \ref{prop: pcl V interseccion completa}  and  the estimate \eqref{eq: estimacion de Ghorpade Lachaud},  we have: if  $\deg(g_j)=0$ for $1\leq i \leq n$, then 
	\begin{equation}\label{estimation pcl V g constante}
	\big||\mathrm{pcl}(V)(\fq)|-p_{t-n}\big | \leq b_{t-n-1}'(t-1,{\bf{d}}) q^{(t-n+1)/2}+9 \cdot 2^{n} (n \cdot d_1+3)^{t+1}q^{(t-n)/2},
	\end{equation} on the other hand if $0 < \deg(g_j)< d_t$, then 
\begin{equation}\label{estimation pcl V g no constante}
\big||\mathrm{pcl}(V)(\fq)|-p_{t-n}\big | \leq b_{t-n-k}'(t-k,{\bf{d}}) q^{(t-n+k)/2}+9 \cdot 2^{n} (n \cdot d_1+3)^{t+1}q^{(t-n+k-1)/2},
\end{equation}
where ${\bf{d}}=(d_1 \klk d_t)$.

Now we estimate the number of $\fq$-rational points of $V^{\infty}=\mathrm{pcl}(V)\cap \{X_0=0\}\subset \Pp^{t-1}$. From Proposition \ref{prop: dimension of pcl en el infinito}, we have that  $V^{\infty}$ is a nonsingular complete intersection. We can apply the following result due to P. Deligne (see, e.g., \cite{De74}): for a nonsingular complete intersection $V \subset \Pp^n$ defined over $\fq$, of dimension $r$ and multidegree ${\bf{d}}=(d_1, \ldots, d_n)$, the following estimate holds:
\begin{equation*}\label{estimacion Deligne}
\big||V(\fq)|-p_r| \leq b'_r(n,{\bf{d}})q^{r/2},
\end{equation*}	
	where $ b'_r(n,{\bf{d}})$ is the rth-primitive Betti number of any nonsingular complete intersection of $\Pp^n$ of dimension $r$ and multidegree ${\bf{d}}$.	
	 Thus
	\begin{equation}\label{estimate V infinito}
	\big||V^{\infty}(\fq)|-p_{t-n-1}\big| \leq b_{t-n-1}'(t-1,{\bfs{d}}) q^{(t-n-1)/2}.
	\end{equation}
	
	If $g_j \in \fq$, $1\leq j \leq n$, from estimates \eqref{estimation pcl V g constante} and \eqref{estimate V infinito} we conclude that 
\begin{align} \label{estimacion: grado igual a cero}
\big||V(\fq)|-q^{t-n}\big|\leq & \big||\mathrm{pcl}(V)(\fq)|-p_{t-n}\big|+\big||V^{\infty}(\fq)|-p_{t-n-1}\big|\\ \notag
\leq &b_{t-n-1}'(t-1,{\bf{d}}) q^{(t-n+1)/2}+9 \cdot 2^{n} (n \cdot d_1+3)^{t+1}q^{(t-n)/2}\\ \notag
& +b_{t-n-1}'(t-1,{\bfs{d}}) q^{(t-n-1)/2}.\notag
\end{align}
If $0<\deg(g_j)\leq d_t$, $1\leq j \leq n$,  from estimates \eqref{estimation pcl V g  no constante} and \eqref{estimate V infinito} we obtain that

\begin{align} \label{estimacion: caso grado mayor a cero}
\big||V(\fq)|-q^{t-n}\big|\leq & \big||\mathrm{pcl}(V)(\fq)|-p_{t-n}\big|+\big||V^{\infty}(\fq)|-p_{t-n-1}\big|\\ \notag
\leq &b_{t-n-k}'(t-k,{\bf{d}}) q^{(t-n+k)/2}+9 \cdot 2^{n} (n \cdot d_1+3)^{t+1}q^{(t-n+k-1)/2}\\ \notag
& +b_{t-n-1}'(t-1,{\bfs{d}})q^{(t-n-1)/2}. \notag
\end{align}
We have the following result.
\begin{theorem}\label{teo: estimacion}
Let  $t,n, d_1,\ldots,d_t, k$ be positive integers such that  $1 \leq k \leq n$, $d_1>\cdots>d_t \geq 2$ and the matrix $A$ satisfies hypothesis $(H)$. Let $g_1, \ldots,g_n \in \fq[X_1, \ldots,X_k]$ such that  $0\leq \deg(g_j)<d_t$ for $1 \leq j \leq n$. Let $|V(\fq)|$ the number of $\fq$--rational points of $V$.
\begin{itemize}
\item  If $g_j \in \fq$ for $1\leq j\leq n$ and $ n \leq t-2$, then  $|V(\fq)|$ satisfies:
$$\big||V(\fq)|-q^{t-n}\big|\leq  q^{\frac{t-n+1}{2}}(6n\cdot d_1)^{t+1}.$$
\item If $0\leq\deg(g_j)<d_t$ for $1\leq j\leq n$, there exists $1\leq i \leq n$ such that $\deg(g_i)>0$, and $ n \leq \frac{t-1}{2}$, then $|V(\fq)|$ satisfies:
$$\big||V(\fq)|-q^{t-n}\big|\leq  q^{\frac{t-n+k}{2}}(6n\cdot d_1)^{t+1}.$$
\end{itemize}
\end{theorem}
\begin{proof}
Suppose that $g_j \in \fq$ for $1\leq j\leq n$, from \eqref{estimacion: grado igual a cero} we need to obtain an upper bound of  the  number $b_{t-n-1}'(t-1,{\bf{d}})$.
From \eqref{cota de Betti} we have that
 $b_{t-n-1}'(t-1,{\bf{d}})\leq \binom{t}{n+1}\cdot (d_1+1)^{t-1}$. On the other hand, taking into account that 
 \begin{equation*}
\binom{t}{n+1} \leq 2^t
 \end{equation*}
% \begin{align*}
% \binom{t}{n+1}\leq & e^{n+1}\bigg(\frac{t}{n+1}\bigg)^{n+1}\\ \notag
% \leq & e^{n+1} \bigg(1+ \frac{t-n-1}{n+1}\bigg)^{n+1} \leq e^{t}.\\ \notag
%    \end{align*}
Then $b_{t-n-1}'(t-1,{\bf{d}})\leq (d_1+1)^{t-1} 2^t$.  
Now, replacing in \eqref{estimacion: grado igual a cero} we obtain that
\begin{align*} \label{estimacion: grado igual a cero}
\big||V(\fq)|-q^{t-n}\big|  
\leq & q^{(t-n-1)/2} (n d_1+3)^{t+1} 2^t\Big( q+ \frac{9}{4}\cdot q^{\frac{1}{2}}+1\Big)\\ \notag
\leq &  2^{t+2} q^{\frac{t-n+1}{2}} (n \cdot d_1+3)^{t+1}  \leq q^{\frac{t-n+1}{2}} (6n \cdot d_1)^{t+1} .\notag
\end{align*}
Now, if $0<\deg(g_j)<d_t$ for $1\leq j\leq n$, the estimate is obtained with similar arguments as above.
\end{proof}
Let $N$ be the number of $\fq$--solutions of the system \eqref{eq: system}. From Theorem \ref{teo: estimacion} we obtain  Theorem \ref{estimate: number of solutions}.
%\begin{corollary} \label{estimate: number of solutions}
%With the same hipotheses as in Theorem \ref{teo: estimacion}, we have that
%\begin{itemize}
%	\item  If $g_i \in \fq$ for $1\leq i\leq n$, then $N$ satisfies:
%	$$\big|N-q^{t-n}\big|\leq  q^{\frac{t-n+1}{2}}(6n\cdot d_1)^{t+1}.$$
% \item If $0\leq\deg(g_j)<d_t$ for $1\leq j\leq n$, there exists $1\leq i \leq n$ such that $\deg(g_i)>0$, and $1 \leq n \leq \frac{t-1}{2}$, satisfies:
%$$\big|N-q^{t-n}\big|\leq  q^{\frac{t-n+k}{2}}(6n\cdot d_1)^{t+1}.$$
%\end{itemize}
%\end{corollary}
\begin{theorem}
If $q>(6nd_1)^{\frac{2t+2}{t-2n}}$ y $ n \leq \frac{t-1}{2}$, then the system \eqref{eq: system} has at least one solution in $\fq^n$.
In particular, if $t$ is sufficiently larger than $2n$, then we can guarantee the existence of an $\fq$--rational solution if $q > (6nd_1)^2$.
 \end{theorem}
 
 \begin{remark} We consider the case $n=1$. Suppose  that $\deg(g_1)=0$, namely, $g_1$ is a nonzero constant in $\fq$. From Theorem \ref{estimate: number of solutions} we have that
 $$\big|N-q^{t-1}\big|\leq  q^{\frac{t}{2}}(6d_1)^{t+1}.$$
We observe that our estimate is of the same order than Weil’s (see, e.g., \cite{Weil49}). On the other hand, this result complements \cite[Theorem 4.1]{PP20} because it does not require that $d_1=\cdots=d_t$. 
 \end{remark} 
 
 \begin{remark} \textbf{Deformed Diagonal Equations}. If $n=1$ we obtain the following deformed diagonal equation:
$$a_{11}X_1^{d_1}+\cdots+ a_{1t}X_t^{d_t}=g_1, $$
with $g_1\in \fq[X_1]$, $0<\deg (g_1)< d_t$, $d_1>\cdots >  d_t \geq 2$ and $\mathrm{char}(\fq)$ does not divide $d_i$ for $1\leq i \leq t$.
From Theorem \ref{estimate: number of solutions} we have that 
$$\big|N-q^{t-1}\big|\leq  q^{\frac{t}{2}}(6d_1)^{t+1}.$$ This result complements \cite[Theorem 4.1] {PP20} in the case that $g$ is an univariate polynomial because the exponents $d_1,\ldots,d_t$ are distinct. 
\end{remark}

\begin{remark}  In \cite{Spackman79} and \cite{Spackman81}, K. W. Spackman studies the number  $N$ of $\fq$--rational solutions of the system \eqref{eq: system} when the polynomials $g_j \in \fq$ for $1\leq j\leq n$. Given $\mu$ a positive integer he defines the parameter $\mu$ of non--singularity. Indeed, for a given $(n\times t)$--matrix in $\fq^{n\times t}$, he says that it is  $\mu$--{\it{weakly non--singular}} if and only if for each natural number $k$ satisfying $\mu \cdot (k-1)+1\leq \min\{t, \mu \cdot (n-1)+1\}$, the matrix has the property that among any $\mu \cdot (k-1)+1$ columns vectors there are at least $k$ $\fq$--linearly independent ones. If $\mu=1$, to be $1$--weakly non--singular is equivalent to satisfies the hypothesis $(H)$.
 In \cite[Theorem 1.1]{Spackman79}  the author proves that if $\mu=1$, $n\geq 2$ then
 $$N=q^{t-n}+ \mathcal{O}(q^{\frac{t-1}{2}}),$$ where the implied constant depends only on $n$, $t$, $d_1,\ldots,d_t$, but it is not explicitly given.
 Theorem \ref{estimate: number of solutions} improved this result in several aspects. Indeed, on one hand, we give an explicit estimate on the number $N$ and we obtain that $N=q^{t-n}+\mathcal{O}(q^{\frac{t-1}{2}-\frac{n-2}{2}})$. On the other hand, the equations can be matched to a non-necessarily constant polynomial.   

In \cite[Theorem 3.2]{Spackman81} the author obtains an explicit estimate on $N$ when $\mu\geq 2$ and $g_j \in \fq$ for $1\leq j \leq n$. More precisely, the following estimate holds
$$|N-q^{t-n}|\leq (d_1-1)\cdots (d_t-1)\cdot (2^{t}-1)\cdot q^{\frac{t+(\mu-2)(n-1)}{2}},$$
 where $n$, $t$ and $\mu$ are positive integers with $\mu\geq 2$ and $t>\mu\cdot(n-1)$. Namely, $N=q^{t-n}+ \mathcal{O}(q^{\frac{t+\epsilon}{2}}),$ $\epsilon>0$.
  Theorem \ref{estimate: number of solutions} improves this result since we have that $N=q^{t-n} + \mathcal{O}(q^{\frac{t-(n-1)}{2}})$, instead of the condition over the coefficients' matrix when $\mu \geq 2$ seems to be weaker than hypothesis $(H)$.
 \end{remark}
%\begin{remark}
%	With same arguments of all this section we have that the second estimate of Corollary \ref{estimate: number of solutions} holds if there exist $j_1,\ldots,j_s$ such that $g_{j_i} \in \fq$ for $1\leq i \leq s<n$.
%\end{remark}
 \subsection{Case $d_1=\cdots=d_t=d\geq 2$} 
We consider the following system:

\begin{equation}\label{eq: system mismo exponente}\left \{\begin{array}{ccl}
a_{11}X_1^{d} & + a_{12}X_2^{d} + \cdots + & a_{1t}X_t^{d} = 0 \\
a_{21}X_1^{d} & + a_{22}X_2^{d} + \cdots + & a_{2t}X_t^{d} =  0 \\
\;\vdots &  &\quad\vdots \\
a_{n1}X_1^{d} &+ a_{n2}X_2^{d} + \cdots +& a_{nt}X_t^{d} = 0
\end{array}\right.
\end{equation}
Suppose that the coefficients' matrix of the above system  satisfies hypothesis $(H)$ and $n\leq t-2$. Let $f_i:=a_{i1}X_1^{d}  + a_{i2}X_2^{d} + \cdots + a_{it}X_t^{d}$, $1\leq i \leq n$. From Theorem \ref{teo: dimension of V}, $V=V(f_1,\ldots,f_n)\subset \Pp^{t-1}$ is a set theoretic projective complete intersection of dimension $t-n-1$. On the other hand, we consider the set $C\subset\A^t $ defined as in \eqref{def C}.
%\begin{equation}\label{def C} C:=\bigg\{{\bf{x}}\in V:\,\,\, \rank \bigg(\frac{\partial f}{\partial {\bf{X}}}\bigg)({\bf{x}}) <n\bigg\}.\end{equation}
From the arguments of Theorem \ref{prop: singular locus g no constante} when $g_i\in \fq$, we have that $C$ is an affine cone of dimension at most $0$.  Then, we deduce that $V \subset \Pp^{t-1}$ is a nonsingular projective variety. From Theorem \ref{theorem: eisenbud 18.15} and $n\leq t-2$, we have that $(f_1,\ldots,f_n)$ is a radical ideal then $V$ is a complete intersection.

Let $\overline{N}$ the number  of  $\fq$--rational projective points of $V$. From \cite{De74} the following estimate holds:
\begin{equation*}
	\big|\overline{N}-p_{t-n-1}\big|\leq   b'_{t-n-1}(t-1,{\bf{d}})q^{\frac{t-n-1}{2}}.
	\end{equation*}
	Since $|V(\fq)|=\overline{N}(q-1)+1$ we conclude that 
	\begin{equation}\label{eq: estimate caso 0}
	\big||V(\fq)|-q^{t-n}\big|\leq   (q-1)2^t(d+1)^{t-1}q^{\frac{t-n-1}{2}}.
	\end{equation}
	In \cite[Theorem II]{T64} A. Tietäväinen studies this type of systems. In concrete he affirms that if $c:=(d,q-1)\geq 3$ and $t\geq 2n(n+\log_{2}(c-1))$ then there exists a non-trivial solution in $\fq^t$. From the above estimate it is easy to obtain the following result.
\begin{proposition} If $q>(4d)^2$ and $t> (n+1)\frac{\log_2(q)}{\log_2(q)-2\log_2(4d)}$ then the system \eqref{eq: system mismo exponente} has at least an $\fq$--rational solution.
\end{proposition}
It is easy to see that 	$1<\frac{\log_2(q)}{\log_2(q)-2\log_2(4d)}\leq 256d^4$ for all $q>16d^2$ while  Tietäväinen's result implies that $n+1\leq n+\log_2(c-1)\leq n+d-1$. Hence, for $n>(4d)^4$, our condition over $t$ is better than Tietäväinen's.
We observe that our result holds for all $d\geq 2$ while Tietäväinen's result holds if $d\geq 3$ and $d$ and $q-1$ have a common divisor greater or equal than $3$.		
Although Tietäväinen's result allows us to consider small values of $q$, sometimes it can not be applied when $q> 16d^2$, for example the case $d=3$ and $q=149$. In conclusion we can say that the Tietäväinen's result and ours are complementary.
	\begin{remark}
 Observe that if $n=1$ we have the equation: 
$$a_{11}X_1^{d}  + a_{12}X_2^{d} + \cdots +  a_{1t}X_t^{d} = 0,$$
Then from \eqref{eq: estimate caso 0} we obtain  the following estimate:
 \begin{equation*}
	\big|N-q^{t-1}\big|\leq   (q-1)2^t(d+1)^{t-1}q^{\frac{t-2}{2}},
	\end{equation*}
which agrees with the well known Weil's estimate for diagonal equations (see \cite{Weil49}). 
\end{remark}
\begin{remark} Suppose now we have the following system:
\begin{equation*}\label{eq: system mismo exponente igual gi}\left \{\begin{array}{ccl}
a_{11}X_1^{d} & + a_{12}X_2^{d} + \cdots + & a_{1t}X_t^{d} = g_1(X_1,\ldots,X_k) \\
a_{21}X_1^{d} & + a_{22}X_2^{d} + \cdots + & a_{2t}X_t^{d} =  g_2(X_1,\ldots,X_k) \\
\;\vdots &  &\quad\vdots \\
a_{n1}X_1^{d} &+ a_{n2}X_2^{d} + \cdots +& a_{nt}X_t^{d} = g_n(X_1,\ldots,X_k),
\end{array}\right.
\end{equation*}
where $g_1, \ldots,g_n \in \fq[X_1, \ldots,X_k]$ are such that  $0\leq \deg(g_j)<d_t$ for $1 \leq j \leq n$. 
Theorem \ref{teo: dimension of V}, Proposition \ref{prop: singular locus g no constante}, Corollary \ref{coro:radical ideal V} and Theorem \ref{teo: V es interseccion completa} work in the case that the exponents $d_i$, $1\leq i \leq n$, are the same.

On the other hand, with similar arguments of Section \ref{geo proyectiva}, Propositions \ref{prop: singular locus pcl V} and \ref{prop: pcl V interseccion completa} and Corollary \ref{corolario: puntos sing en el infinito} hold in this case. 
Then, from Theorem \ref{estimate: number of solutions}, $N$ satisfies the following estimates:
\begin{itemize}
	\item If $g_j \in \fq$, $1\leq j\leq n$ and $n\leq t-2$,  
	$$\big|N-q^{t-n}\big|\leq  q^{\frac{t-n+1}{2}}(6n\cdot d)^{t+1}.$$ 
 \item If $0\leq\deg(g_j)<d_t$, $1\leq j\leq n$, there exists $1\leq i \leq n$ such that $\deg(g_i)>0$ and $ n \leq \frac{t-1}{2}$,  
$$\big|N-q^{t-n}\big|\leq  q^{\frac{t-n+k}{2}}(6n\cdot d)^{t+1},$$ 
\end{itemize}
\end{remark}
 \subsection{Generalized Markoff-Hurwitz-type systems}
A concrete example of a system of the form  \eqref{eq: system} are the Markoff-Hurwitz systems. The Markoff-Hurwitz equations have been very well studied (see, e.g., \cite{Mo63}, \cite{JGC20} and \cite{PP20}) but, there are not results in the literature about this type of systems.

Let  $t,n, d_1,\ldots,d_t$ be positive integers such that $ n <\frac{t-1}{2}$, $d_1>\cdots>d_t \geq 2$, and $\mathrm{char}(\fq)$ does  not divide $d_i$ for $1\leq i \leq t$. Let $c_{ij}$ be positive integers such that $1 \leq i,j \leq n$ and $c_{j1}+\cdots+c_{jn} <d_t$, $1 \leq j \leq n.$ We consider the following system of $n$  generalized Markoff-Hurwitz-type equations with $t$ unknowns over $\fq$
\begin{equation}\label{eq: system markoff hurwitz}\left \{\begin{array}{ccl}
a_{11}X_1^{d_1} & + a_{12}X_2^{d_2} + \cdots + & a_{1t}X_t^{d_t} +a_1=b_1X_1^{c_{11}}\ldots x_n^{c_{1n}} \\
a_{21}X_1^{d_1} & + a_{22}X_2^{d_2} + \cdots + & a_{2t}X_t^{d_t}+a_2 =b_2X_1^{c_{21}}\ldots X_n^{c_{2n}} \\
\;\vdots &  &\quad\vdots \\
a_{n1}X_1^{d_1} &+ a_{n2}X_2^{d_2} + \ldots +& a_{nt}X_t^{d_t}+ a_n = b_nX_1^{c_{n1}}\ldots X_n^{c_{nn}},
\end{array}\right.
\end{equation}
where there exists $a_j \neq 0$ with $1\leq j \leq n$ and $b_1, \ldots, b_n \neq 0$.
Assume that  the coefficients' matrix of the above system satisfies the hypothesis $(H)$.
Denote by $N$ the number of $\fq$--rational solutions of \eqref{eq: system markoff hurwitz}.  Let $g_j:=b_jX_1^{c_{j1}}\ldots X_n^{c_{jn}}-a_j$, $1 \leq j \leq n$. Since $\deg(g_j) < d_t$, $1 \leq j \leq n$, from Theorem \ref{estimate: number of solutions} we obtain the following result.
\begin{theorem} \label{teo: system markoff hurwitz}
	With the same hypotheses as above, $N$ satisfies the following estimate:  
	\begin{equation*} 
\big|N-q^{t-n}\big|\leq  q^{\frac{t}{2}}(6n\cdot d_1)^{t+1}.
	\end{equation*}
\end{theorem}
In what follows we obtain sufficient conditions for the existence of an $\fq$--rational solution with nonzero coordinates namely, with coordinates in $\fq^{*}$. Denote by  $N^*$  the number of this type of solutions of  \eqref{eq: system markoff hurwitz}. Let $N^{=}$ be the number of $\fq$-rational solutions of \eqref{eq: system markoff hurwitz} with at least one coordinate equals to zero. Note that $N^* =N- N^{=}$. 
 
By the inclusion-exclusion principle we obtain that
\begin{equation}\label{Pcio inclusion exclusion}
N^{=}=\sum_{i=1}^t (-1)^{i+1} \binom{t}{i} N_i.
\end{equation}

We shall need the following estimate on the number of $\fq$--rational solutions of \eqref{eq: system markoff hurwitz} with exactly $i$ coordinates equal to zero. We denote this number by  $N_i$. 

\begin{proposition} \label{estimation Ni}  With the same hypotheses as above, the number $N_i$ satisfies the following estimate:
	
	If $1\leq i \leq t-2n-1$, then
	\begin{equation} \label{case 1}
	|N_i-q^{t-i-n}| \leq q^{\frac{t-i}{2}}(6n\cdot d_1)^{t-i+1}.
	\end{equation}
	If $t-2n \leq i \leq  t-n$, then 
	\begin{equation*}\label{case 2}
	N_i \leq d_1^n q^{t-n-i}
	\end{equation*}
	If $t-n+1 \leq i \leq t$, then $N_i \leq d_1^{n}$.

\end{proposition}
\begin{proof} 
Suppose that $1 \leq i \leq t-2n-1$. We observe that $N_i$ is the number of $\fq$--rational solutions of  a system of $n$ deformed diagonal equations with $t-i$ unknowns. The coefficients' matrix of the system satisfies hypothesis $(H)$.
%$$a_{j1}x_1^{d_1}+ a_{j2}x_2^{d_2} + \cdots +  a_{jt}x_t^{d_t}=g_j,$$ 
%where $g_j:=b_jx_1^{c_{j1}}\ldots x_n^{c_{jn}}-a_j$,  for $1 \leq j \leq n$.
Then we deduce \eqref{case 1} from Theorem \ref{estimate: number of solutions}.

%Suppose now that $n \leq i \leq t-n-2$. In this case $N_i$ is the number of $\fq$--rational solutions of the a system of $n$ diagonal equations with $t-i$ unknowns that are matched to constants.
%$$a_{j1}x_1^{d_1}+ a_{j2}x_2^{d_2} + \cdots +  a_{jt}x_t^{d_t}=-a_j,$$
%where $1 \leq j \leq n$. 
%From Corollary \ref{estimate: number of solutions} we deduce \eqref{case 2}.

Suppose now $t-2n\leq i\leq t-n$. 
%from \eqref{} we deduce that $N_i \leq d_1^{n+1} q$. 
In this case, $N_i$ is the number of $\fq$--rational solutions of a system of $n$ deformed diagonal equations with $t-i\geq n$ unknowns.
%$$a_{j1}x_1^{d_1}+ a_{j2}x_2^{d_2} + \cdots +  a_{jt}x_t^{d_t}=-a_j,$$
We observe that, since the coefficients' matrix of the system satisfies hypothesis $(H)$ we can follow the same arguments of the proofs of Theorems \ref{teo: dimension of V} and \ref{teo: V es interseccion completa}.  Hence, if $V_i \subset \A^{t-i}$ is the set of solutions of the corresponding system then $V_i$ is an $\fq$--affine complete intersection of dimension $t-i-n$ and $\deg(V_i)\leq d_1^n$. Finally, from \eqref{eq: upper bound -- affine gral}, we have that $N_i \leq d_1^n \cdot q^{t-i-n}.$

Let $t-n+1\leq i \leq t$. In this case $n$, the number of equations, is greater than the number of unknowns $t-i$. Since the coefficients' matrix of the system \eqref{eq: system}, satisfies hypothesis $(H)$, then, the coefficients' matrix of the system of this case, has rank $t-i$. So the set of solutions has dimension at most zero. Hence, from  \eqref{eq: upper bound -- affine gral}, $N_i\leq d_1^n$.
\end{proof}

%Let $N^*$ be the number of $\fq$-rational solutions of \eqref{eq: system markoff hurwitz} with nonzero coordinates and let $N^{=}$ be the number of $\fq$-rational solutions of \eqref{eq: system markoff hurwitz} with at least one coordinate equal to zero. Note that $N^* =N- N^{=}$. 
%By the inclusion-exclusion principle we obtain
%\begin{equation}\label{Pcio inclusion exclusion}
%N^{=}=\sum_{i=1}^t (-1)^{i+1} \binom{t}{i} N_i.
%\end{equation}

From \eqref{Pcio inclusion exclusion} and taking into account that $N^{*}=N-N^{=}$, we have that
\begin{align*}
N^*&=N+ \sum_{i=1}^t (-1)^{i} \binom{t}{i} N_i=N+\!\! \sum_{i=1}^{\!t-2n-1}\! (-1)^{i} \binom{t}{i} N_i+\sum_{i=t-2n}^t (-1)^{i} \binom{t}{i} N_i\\
&=N+ \!\! \sum_{i=1}^{t-2n-1} (-1)^{i}\! \binom{t}{i} \!(N_i-q^{t-n-i})+ \!\!\sum_{i=1}^{t-2n-1} (-1)^{i} \! \binom{t}{i} q^{t-n-i}+\sum_{i=t-2n}^t (-1)^{i}\! \binom{t}{i} N_i\\
&=N\!-\!q^{t-n}+ \!\!\!\sum_{i=1}^{t-2n-1}(-1)^i \binom{t}{i} \! (N_i-q^{t-n-i})\!\!+ \!\! \sum_{i=0}^{t-2n-1}(-1)^i\! \binom{t}{i}q^{t-n-i}+\!\!\sum_{i=t-2n}^t \!(-1)^{i} \binom{t}{i} N_i.
\end{align*}
Thus,  we deduce that
\begin{equation*}
N^*\!\!-\!\!\sum_{i=0}^{t-2n-1}\!\!(-1)^i \binom{t}{i}q^{t-n-i}\!\!=(N-q^{t-n})\!\!+ \!\!\sum_{i=1}^{t-2n-1}\!\!(-1)^i \! \binom{t}{i} (N_i-q^{t-n-i})\!\!+ \!\!\sum_{i=t-2n}^t (-1)^{i}\! \binom{t}{i} N_i.
\end{equation*}
Therefore, from Theorem \ref{teo: system markoff hurwitz} and Proposition \ref{estimation Ni}:
\begin{align*}\label{a distinto cero}
\Big|N^*  \!\!\!-\!\! \!\!\!\sum_{i=0}^{t-2n-1} \!\!(-1)^i\!\binom{t}{i}q^{t-n-i}\Big| & \!\leq\! |N-q^{t-n}|+\sum_{i=1}^{t-2n-1}\binom{t}{i}|N_i-q^{t-n-i}| 
+ \sum_{i=t-2n}^{t}\binom{t}{i}N_i\\
& \leq \! (6nd_1)^{t+1}\!\Big(q^{\frac{t}{2}} \!\! +\!\!\! \sum_{i=1}^{t-2n-1}\!\!\!\binom{t}{i} q^{\frac{t-i}{2}} \Big)+  \!\!\!\! \sum_{i=t-2n}^{t-n}\!\!N_i\binom{t}{i}\!\! \!+ \!\!\!\!\sum_{i=t-n+1}^{t}\!\!\!\!\!N_i\binom{t}{i} \\
%& \leq \!(6nd_1)^{t+1}\Bigg(q^{\frac{t}{2}}\! + \!\!\sum_{i=1}^{t-n-2}\binom{t}{i} q^{\frac{t-i}{2}}\Bigg) + 2^t d_1^{n-1} q + 2^t\\
& \leq\!  (6nd_1)^{t+1}\Bigg(q^{\frac{t}{2}} \!\! +  \!\! 2^{t}q^{\frac{t-1}{2}}\Bigg)  \!\!+ 2^t d_1^{n} (q^n+1)\\
& \leq \!(6nd_1)^{t+1}q^{\frac{t-1}{2}}(2^t+q^{\frac{1}{2}}) \!\! +  \!\!2^{t+1} d_1^{n}q^{\frac{t-1}{2}}\\
& \leq \! 2^{t+2}(6nd_1)^{t+1}q^{\frac{t}{2}}\\
&\leq\!  (15nd_1)^{t+1}q^{\frac{t}{2}}
\end{align*}
We have proved the following result.
\begin{proposition}\label{estimation N*}
If $q>2$, $1\leq n < \frac{t-1}{2}$, $d_1\geq 2$ and $\mathrm{char}(\fq)$ does  not divide $d_i$ for $1\leq i \leq t$. Then, the number $N^*$ of $\fq$--rational solutions of \eqref{eq: system markoff hurwitz} with nonzero coordinates  satisfies the following estimate:
	$$\Bigg|N^*-\bigg(\frac{(q-1)^t}{q^n}-\sum_{i=t-2n}^{t}(-1)^i \binom{t}{i}q^{t-n-i}\bigg)\Bigg| \leq  (15nd_1)^{t+1}q^{\frac{t}{2}}.$$
\end{proposition}
In \cite{PP20} we study  the following  Markoff-Hurwitz's equation:
$$a_{1}X_1^{d_1} + a_{2}X_2^{d_1} + \cdots +  a_{t}X_t^{d_1} +a=bX_1^{c_{1}}\ldots X_t^{c_{t}},$$
where $a_i\in \fq$, $1\leq i \leq t$ and $a,b\in \fq\setminus\{0\}$.
 Indeed, in \cite[Proposition 4.7]{PP20}, we have that, if $q>2$, then
$$\bigg|N^*-\frac{(q-1)^t-(-1)^t}{q}\bigg| \leq 7(12d_1)^{t}q^{\frac{t}{2}}.$$
 In particular, Proposition \ref{estimation N*}  provides an estimate in the case $n=1$ and $c_{n+1}=\cdots=c_{t}=0$.  The error term of both estimates is of order of $\mathcal{O}(q^{t/2})$ but, in Proposition \ref{estimation N*} we obtain two new terms in the asymptotic development of $N^*$ in terms of $q$. Indeed, we have that $N^*=\frac{(q-1)^t-(-1)^t}{q}+(-1)^{t}t+(-1)^{t-1}\frac{t(t-1)}{2}q+ \mathcal{O}(q^{t/2})$.

Now, we provide an existence result for $\fq$--rational solutions with nonzero coordinates.  
%$$\frac{(q-1)^t}{q^n}= \sum_{i=0}^{t-2n-1}(-1)^i \binom{t}{i}q^{t-n-i} + \sum_{i=t-2n}^{t}(-1)^i \binom{t}{i}q^{t-n-i}. $$ Thus, we obtain that if $q >2$, then
%$\sum_{i=0}^{t-2n-1}(-1)^i \binom{t}{i}q^{t-n-i} \geq \frac{(q-1)^t}{q^n}- 2^tq^{-n}.$
Suppose that $q>2$. From  the above proposition we deduce that
\begin{equation*}
N^* \geq \frac{(q-1)^t}{q^n}-\sum_{i=t-2n}^{t}(-1)^i \binom{t}{i}q^{t-n-i}-(15nd_1)^{t+1}q^{\frac{t}{2}}.
\end{equation*}
It is easy to prove that  $ (-1)^t\sum_{i=t-2n}^{t}(-1)^i \binom{t}{i}q^{t-n-i}\geq 0$ if $q\geq\frac{2n}{t-2n+1}.$
Then, by elementary calculations we have that
\begin{align*}
N^*& \geq \frac{(q-1)^t}{q^n}- 2(15nd_1)^{t+1}q^{\frac{t}{2}}\\
%& \geq \frac{((q-)^t-2^t)}{q^n} - 2(15nd_1)^{t+1}q^{\frac{t}{2}}\\
%& \geq \frac{(q-3)^t}{q^n}-2(15nd_1)^{t+1}q^{\frac{t}{2}}\\
& \geq \frac{q^{t-n}}{2^t}-2(15nd_1)^{t+1}q^{\frac{t}{2}}\\
& \geq q^{\frac{t}{2}}\Big( \frac{q^{\frac{t-2n}{2}}}{2^t}- 2(15nd_1)^{t+1}\Big)
\end{align*}
Therefore, \eqref{eq: system markoff hurwitz} has at least one solution in $\fq^t$ with nonzero coordinates if 
$$\frac{q^{\frac{t-2n}{2}}}{2^t}- 2(15nd_1)^{t+1}>0.$$
That is,
$$q^{\frac{t-2n}{2}}> (30nd_1)^{t+1}.$$

Finally we obtain the following result.
\begin{proposition}\label{condition existence eq M H}
	If   $q >(30nd_1)^\frac{2t+2}{t-2n}$ and $n<\frac{t-1}{2}$ then the system \eqref{eq: system markoff hurwitz} has at least one solution in $(\fq^*)^t$. In particular, if $t$ is sufficiently larger than $2n$, then we can guarantee the existence of an $\fq$--rational solution if $q > (30nd_1)^2$.
\end{proposition}

\section{Generalization:  variants of systems of diagonal equations} 
Let  $t,n, d_1,\ldots,d_t, k$ be positive integers such that $n \leq \frac{t-1}{2}$ ,  $1 \leq k \leq n$ and $d_1>\cdots>d_t \geq 2$, and $\mathrm{char}(\fq)$ does  not divide $d_i$ for $1\leq i \leq t$.  Let be $h_1, \ldots,h_t \in \fq[T]$ with $\deg(h_i)=d_i$ for $1\leq i \leq n$. Let $X_1, \ldots, X_t$ be indeterminates over $\fq$ and let $g_1, \ldots,g_n \in \fq[X_1, \ldots,X_k]$ such that $g_j \in \fq$ for $1 \leq j \leq n$ or $0\leq \deg(g_j)<d_t$ for $1 \leq j \leq n$ and there exists $i$ such that $\deg(g_i)>0$. 

We consider the following system of $n$ variants of diagonal equations with $t$ unknowns
\begin{equation}\label{eq: system h}\left \{\begin{array}{ccl}
a_{11}h_1(X_1) & + a_{12}h_2(X_2) + \cdots + & a_{1t}h_t(X_t) = g_1(X_1, \ldots,X_k) \\
a_{21}h_1(X_1) & + a_{22}h_2(X_2) + \cdots + & a_{2t}h_t(X_t) = g_2(X_1, \ldots,X_k) \\
\;\vdots &  &\quad\vdots \\
a_{n1}h_1(X_1) &+ a_{n2}h_2(X_2) + \cdots +& a_{nt}h_t(X_t) = g_n(X_1, \ldots,X_k)
\end{array}\right.
\end{equation}
Assume that  the coefficients' matrix of the above system satisfies hypothesis $(H)$.
This is a system of Carlitz's equations. These equations  has been defined in \cite{Ca53}. There, the author  provides an estimate for the case $n=1$ and in \cite{PP20} we improve his results in several aspects.

Let $V:=V(f_1 \klk f_n) \subset \A^t$ be the $\fq$--affine variety defined by $f_i:=a_{i1}h_1(X_1) + a_{i2}h_2(X_2) + \cdots + a_{it}h_t(X_t) - g_i(X_1, \ldots,X_k)$, for $ 1 \leq i \leq n$.
With the same arguments of Theorem \ref{teo: dimension of V}, we obtain that $V \subset \A^t$ is a set-theoretic complete intersection of dimension $t-n$.
We consider the set $C$ as in \eqref{def C}.

Let  $1\leq k\leq n$ and $ n \leq \frac{t-1}{2}$. Suppose that $g_j \in \fq$ or $  0\leq \deg(g_j) < d_t$  and there exists $g_i$ such that $\deg(g_i)>0$. % Assume that $A$, the coefficients' matrix of the system \eqref{eq: system}, satisfies the hypothesis $(H)$.
Observe that
\begin{equation*} \label{eq:matriz jacobiana factorizada case gral}
\frac{\partial f}{\partial \bf{X}}:= \left(
\begin{array}{c|c}
M_1 & M_2
\end{array}
\right)
\end{equation*}
where $M_1$ is a $(n \times k)$--matrix defined by
$$M_1:=\left(
\begin{array}{ccccccc}
a_{11}h_1'(X_1)+\frac{\partial g_1}{\partial X_1} & \cdots &  & \cdots &   a_{1k}h_k'(X_k)+ \frac{\partial g_1}{\partial X_k}
\\
\vdots &  \vdots  &   & \vdots &   \vdots
\\
a_{n1}h_1'(X_1)+\frac{\partial g_n}{\partial X_1}& \cdots &    & \cdots &   a_{nk}h_k'(X_k)+\frac{\partial g_n}{\partial X_k}
\end{array} \right)$$
and $M_2$ is a  $n \times(t-k)$--matrix defined by
$$M_2:=\left(
\begin{array}{ccccccc}
a_{1k+1}h_{k+1}'(X_{k+1})& \cdots &  & \cdots &   a_{1t}d_th_t'(X_t)
\\
\vdots &  \vdots  &   & \vdots &   \vdots
\\
a_{n{k+1}}h_{k+1}'(X_{k+1})& \cdots &    & \cdots &   a_{nt}h_t'(X_t)
\end{array} \right)$$

%Let $C$ be the set defined in \eqref{def C}.
\begin{proposition} The dimension of $C$ is at most $k-1$ if $\deg(g_j)\geq 0$ for $1\leq j \leq n$ and there exists $i$ such that $\deg{g_i}>0$. On the other hand, this dimension is at most $0$ if $g_i \in \fq$ for $1\leq i \leq n$. In particular, the dimension of the singular locus of $V$ is at most $k-1$ or $0$ respectively.
\end{proposition}
\begin{proof}
	Let ${\bf{x}}\in C$. We observe that $M_2$ satisfies 
$$M_2:=\left(
\begin{array}{ccccccc}
a_{1k+1}& \cdots &  & \cdots &   a_{1t}
\\
\vdots &  \vdots  &   & \vdots &   \vdots
\\
a_{n{k+1}}& \cdots &    & \cdots &   a_{nt}
\end{array} \right) \cdot\left(
\begin{array}{ccccccc}
h_{k+1}'(X_{k+1})& \cdots &  & \cdots &   0
\\
\vdots &  \vdots  &   & \vdots &   \vdots
\\
0& \cdots &    & \cdots &  h_t'(X_t)
\end{array} \right)$$	
	
From hypothesis $(H)$ we have that the diagonal matrix of right side can not have $n$ columns nonzero. Then, we deduce that the number of zero columns is at least $t-n-k+1$. Suppose that ${\bf{x}} \in C$ is such that $h_{k+1}'(x_{k+1})=0 \klk h_{t-n+1}'(x_{t-n+1})=0$. Then, we obtain that the coordinates $X_{k+1} \klk X_{t-n+1}$ of ${\bf{x}}$ take finite values in $\cfq$. Then, we deduce that $C$ is contained in a finite union  of $\cfq$--linear varieties of dimension $n+k-1$. From hypothesis $(H)$, the intersection of each of these linear variety with $V$ is a subvariety of $V$ of dimension $k-1$, if $\deg(g_j)>0$ for $1 \leq j \leq n$, and the dimension is $0$, if $g_j \in \fq$ for $1 \leq j \leq n$.

\end{proof}

For $\mathrm{pcl}(V)\subset \Pp^t$ holds the same results of the Section \ref{geo proyectiva}. Therefore, we have the following estimates on $N$, the number  of $\fq$--rational solutions of the system defined in \eqref{eq: system h}.
\begin{theorem} \label{estimate: number of solutions h}
$N$ satisfies:
	\begin{itemize}
		\item  If $g_j\in \fq$ for $1\leq j\leq n$ and $n\leq t-2$ then $N$ satisfies:
		$$\big|N-q^{t-n}\big|\leq  q^{\frac{t-n+1}{2}}(6n\cdot d_1)^{t+1}.$$
		\item If $0\leq \deg(g_j)<d_t$ for $1\leq j\leq n$ and there exists $i$ such that $\deg(g_i)>0$, $n\leq \frac{t-1}{2}$ then $N$ satisfies:
		$$\big|N-q^{t-n}\big|\leq  q^{\frac{t-n+k}{2}}(6n\cdot d_1)^{t+1}.$$
	\end{itemize}
\end{theorem}
\begin{corollary}
	If $t > > 2n$ and $q > (6nd_1)^2$, then the system \eqref{eq: system h} has at least one solution in $\fq^n$.
\end{corollary}
\begin{remark} In \cite[Theorem 2]{T65 A} Tietäväinen deduces an explicit estimate on the number of $\fq$--rational solutions of \eqref{eq: system h} where $g_j=0$ for all $1\leq j \leq n$. He establishes that $N=q^{t-n} + \mathcal{O}(q^{t/2}).$
\end{remark}
\begin{remark}
In \cite{PP20} we study the Carlitz's equations (see \cite{Ca53}, for more details about these equations).
Let  $d,t$ be positive integers with $d\geq 2$ and $t\geq 3$. Let $h_i=a_{d,i} T^d+\cdots +a_{0,i} \in \fq[T]$, with $\deg(h_i)=d$, $1\leq i\leq t$. Let $g\in \fq[X_1,\ldots,X_t]$ such that $\deg (g)< d$. Suppose that $\mathrm{char}(\fq)$  does not divide $d$. We consider the following Carlitz's equation:
\begin{equation*}\label{eq. carlitz}
h_1(X_1)+\cdots +h_t(X_t)=g.
\end{equation*} 
We obtain an explicit estimate on the number $N$ of $\fq$--rational solutions of Carlitz's equations. Indeed, we have that  
\begin{equation}\label{estimation Nbis}
	\big|N-q^{t-1}\big|\leq   q^{(t-1)/2}\big(2(d-1)^{t-1}q^{1/2}+6(d+2)^t\big).
	\end{equation}
	The result of Theorem \ref{estimate: number of solutions h} complements the estimate \eqref{estimation Nbis} when $g$ is an univariate polynomial and the degrees of the polynomials $h_i$ are distinct.
\end{remark}

\subsection{Systems of Dickson's equations}
These systems are a particular case of systems of the form \eqref{eq: system h}.
Let $d \in \mathbb{N}$ and $a \in \fq$. The Dickson polynomials over $\fq$ of degree $d$ with parameter $a$:
$$D_d(X,a)=\sum_{i=0}^ {\lfloor d/2 \rfloor} \frac{d}{d-i} \binom{d-i}{i} (-a)^i X^{d-2i} .$$
Dickson polynomials have been extensively studied  because they play very important roles in both theoretical work as well as in various applications (see, \cite[Chapter 7]{MuPa13}).  The set of $\fq$--rational solution of  Dickson's equations has been very well studied in the literature (see, \cite{PP20}, \cite{ChMuWa08}).  However, there are less results concerning the set of solutions of systems of equations given by Dickson polynomials.

Let  $t,n, d_1,\ldots,d_t, k$ be positive integers such that $ n \leq \frac{t-1}{2}$ ,  $0 \leq k \leq n$ and $d_1>\cdots>d_t \geq 2$, and $\mathrm{char}(\fq)$ does  not divide $d_i$ for $1\leq i \leq t$.  Let be $D_{d_1}(T,a_1), \ldots, D_{d_t}(T, a_t) \in \fq[T]$ with $a_1,\ldots,a_t \in \fq$. Let $X_1, \ldots, X_t$ be indeterminates over $\fq$ and let $g_1, \ldots,g_n \in \fq[X_1, \ldots,X_k]$ such that $g_j \in \fq$ for $1 \leq j \leq n$ or $0 \leq \deg(g_j)<d_t$ for $1 \leq j \leq n$ and there exists $i$ such that $\deg(g_i)>0$. 

We consider the following system of $n$ Dickson's equations with $t$ unknowns
\begin{equation*}\left \{\begin{array}{ccl}
a_{11}D_{d_1}(X_1,a_1) & + a_{12}D_{d_2}(X_2,a_2) + \cdots + & a_{1t}D_{d_t}(X_t,a_t) = g_1(X_1, \ldots,X_k) \\
a_{21}D_{d_1}(X_1,a_1) & + a_{22}D_{d_2}(X_2,a_2) + \cdots + & a_{2t}D_{d_t}(X_t,a_t) = g_2(X_1, \ldots,X_k) \\
\;\vdots &  &\quad\vdots \\
a_{n1}D_{d_1}(X_1,a_1) &+ a_{n2}D_{d_2}(X_2,a_2) + \cdots +& a_{nt}D_{d_t}(X_t,a_t) = g_n(X_1, \ldots,X_k)
\end{array}\right.
\end{equation*}
%Let $A=[a_{ij}]\in \fq^{n \times t}$ be the coefficients matrix of the above system.  Assume that $A$ satisfies  hypothesis $(H)$.
From Theorem \ref{estimate: number of solutions h} we obtain an estimate on the number of $\fq$--solutions of this type of systems.

\begin{remark} We observe that similar arguments prove that  Theorem \ref{estimate: number of solutions h} holds if $d_1=\cdots=d_t$.
\end{remark}
%\begin{remark}
% With same arguments we have that the second estimate of Theorem \ref{estimate: number of solutions h} holds if there exist $j_1,\ldots,j_s$ such that $\deg(g_{j_i})=0$ for $1\leq i \leq s$.
%\end{remark}

%----------------------------------------------------------------------------------
%----------------------------------------------------------------------------------
%----------------------------------------------------------------------------------
\section{Applications}
\subsection{Generalized Waring's problems over finite fields}

One of the most important questions in number theory is to find properties on a system of equations that guarantee solutions over a field, for example, the so called generalized Waring’s problem ( see, e.g. \cite{Cao2016}, \cite{CRGF08} and \cite{T65}). 
 Let be the following system over $\fq$ with $n$ equations and $t$ indeterminates
 \begin{equation}\label{eq: system waring}\left \{\begin{array}{ccl}
a_{11}X_1^{d_1} & + a_{12}X_2^{d_2} + \cdots + & a_{1t}X_t^{d_t} = b_1\\
a_{21}X_1^{d_1} & + a_{22}X_2^{d_2} + \cdots + & a_{2t}X_t^{d_t} =  b_2 \\
\;\vdots &  &\quad\vdots \\
a_{n1}X_1^{d_1} &+ a_{n2}X_2^{d_2} + \cdots +& a_{nt}X_t^{d_t} = b_n,
\end{array}\right.
\end{equation}
where the coefficients' matrix of the system satisfies the hypothesis (H), $2\leq d_t<\cdots < d_1$, $\mathrm{char}(\fq)$ does not divide $d_i$ for $1\leq i \leq t$.

Waring's problem consists in finding  $\gamma(t,n,q,d_1)$ the least number of variables $t$ such that  \eqref{eq: system waring} has solution in $\fq^t$ for every pair $(b_1,\ldots,b_n)\in \fq^n$.
From Theorem \ref{estimate: number of solutions} we have that $N$, the number of $\fq$--rational solutions of \eqref{eq: system waring}, satisfies that
$$N\geq q^{\frac{t-n+1}{2}} \bigg( q^{\frac{t-n-1}{2}}- (6nd_1)^{t+1}\bigg).$$
Then $N>0$ provided that $q^{\frac{t-n-1}{2}}-(6nd_1)^{t+1}>0$, namely $q^{\frac{t-n-1}{2}}> \big(6nd_1\big)^{t+1}$.
Now if $q>( 6nd_1)^2$ then, the last condition is equivalent to
$$t> \frac{\log(6 nd_1\cdot q^{\frac{n+1}{2}})}{\log (\frac{q^{1/2}}{6nd_1})}.$$
Then, if $q>( 6nd_1)^2$ we obtain that $$\gamma(t,n,q,d_1)\leq \bigg \lceil\frac{\log(6nd_1\cdot q^{\frac{n+1}{2}})}{\log (\frac{q^{1/2}}{6 nd_1})} \bigg\rceil. $$

We observe that $h(q):= \frac{\log(6 nd_1\cdot q^{\frac{n+1}{2}})}{\log (\frac{q^{1/2}}{6nd_1})}$ is a decreasing function and $\lim_{q\rightarrow \infty}h(q)=n+1.$ Therefore $h(q)\geq n+1$ if $q>(6nd_1)^2$.
Then we deduce that if $q$ sufficiently large, %\eqref{eq: system waring} has $\fq$--rational solutions for every pair $(b_1,\ldots,b_n)\in \fq^n$,
$\gamma(t,n,q,d_1)\leq n+2.$% Moreover, note in particular that solutions with small number of variables require large values of  $q$.  
\subsection{Distribution of solutions of systems of congruences equations modulo a prime number}
In this section we apply our estimates to obtain asymptotic formulas for the distribution of simultaneous solutions to congruences modulo $p$, a prime number.  This is a well studied problem, see, for example \cite{Spackman81} and \cite{T67}.

Let  $t,n, d_1,\ldots,d_t$ be positive integers such that $1\leq n \leq \frac{t-1}{2}$  and $d_1>\cdots>d_t \geq 2$, and $p$ does  not divide $d_i$ for $1\leq i \leq t$.% Let $X_1, \ldots, X_t$ be indeterminates over $\fq$ and let $g_1, \ldots,g_n \in \fq[X_1, \ldots,X_k]$ such that $g_j\in \fq$ for $1 \leq j \leq n$ or $0<\deg(g_j)<d_t$ for $1 \leq j \leq n$. 

We consider the following systems of  congruences equations
\begin{equation}\label{eq: system congruencia}\left \{\begin{array}{ccl}
a_{11}X_1^{d_1} & + a_{12}X_2^{d_2} + \cdots + & a_{1t}X_t^{d_t} \equiv 0 \pmod{p} \\
a_{21}X_1^{d_1} & + a_{22}X_2^{d_2} + \cdots + & a_{2t}X_t^{d_t}\equiv 0 \pmod{p} \\
\;\vdots &  &\quad\vdots \\
a_{n1}X_1^{d_1} &+ a_{n2}X_2^{d_2} + \cdots +& a_{nt}X_t^{d_t} \equiv 0 \pmod{p}
\end{array}\right.
\end{equation}
Assume that the coefficients' matrix, satisfies the hypothesis $(H)$.
\smallskip
From Theorem \ref{estimate: number of solutions} we have an estimate on  $N_p$, the number of solutions in $[0,p-1]^t$.
Indeed, the following estimate holds: 
\begin{equation}\label{estimate: N_p}
\big|N_p-p^{t-n}\big|\leq  p^{\frac{t-n+1}{2}}(6n\cdot d_1)^{t+1}.
\end{equation}
Let be $m<\frac{p}{2^{t-n}}$ and suppose that $p>2^{\frac{2t-2n}{t-n-1}}$. Our purpose is to obtain an estimate on $N_m$, the number  of solution in $[0,p-m-1]^t$. Let $S_1$ and $S_2$ the following intervals in $\Z$: $S_1=[0,p-m-1]$ and $S_2=[p-m,p-1]$.
From the well known Zippel--Schwartz Lemma (see, e.g., \cite{GaGe99}) we have that
$$|V\cap S_1^t|\leq d_1^t(p-m)^{t-n},\quad |V\cap S_2^t|\leq d_1^t m^{t-n},$$
where $V \subset \A^t$ is the $\fp$--variety defined by the polynomials $ f_j:=a_{j1}X_1^{d_1}+a_{j2}X_2^{d_2}+\cdots+a_{jt}X_t^{d_t},\,\,1 \leq j \leq n$.
Then 
\begin{align*}
\big||V\cap S_1^t|-(p-m)^{t-n}\big|&\leq \big||V\cap \fp^{t}| - p^{t-n}\big| + |V\cap S_2^t| +(p^{t-n}-(p-m)^{t-n})\\
& \leq  p^{\frac{t-n+1}{2}}(6n\cdot d_1)^{t+1}+ d_1^t m^{t-n} + m(t-n)p^{t-n-1}\\
& \leq  p^{t-n-1}(6n\cdot d_1)^{t+1}m(t-n)
\end{align*}
Finally the number of solution in the $t$--cube $[0,p-m-1]^t$ satisfies $N_m=(p-m)^{t-n} + \mathcal{O}(m\cdot p^{t-n-1})$ with  $m<\frac{p}{2^{t-n}}$ and  $p>2^{\frac{2t-2n}{t-n-1}}$.


\begin{thebibliography}{99}
	
%\bibitem{AuRo10} Y. Aubry and F. Rodier, \emph{Differentially 4-uniform functions}, Arithmetic, geometry, cryptography and coding theory 2009, 1--8, Contemp. Math., 521, Amer. Math. Soc., Providence, RI, 2010. 
	
%	\bibitem[BBR15]{BaBaRo15}
%	E.~Bank, L.~{Bary-Soroker}, and L.~Rosenzweig, \emph{Prime
%		polynomials in short
%		intervals and in arithmetic progressions}, Duke Math. J. \textbf{164} (2015),
%	no.~2, 277--295.
%	
%	\bibitem[Ben12]{Benoist12}
%	O.~Benoist, \emph{Degr\'es d'homog\'en\'eit\'e de l'ensemble des
%		intersections
%		compl\`etes singuli\`eres}, Ann. Inst. Fourier (Grenoble) \textbf{62} (2012),
%	no.~3, 1189--1214.
%\bibitem{AdSp87} A. Adolphson and S. Sperber, \emph{Exponential sums and Newton polyhedra}, Bull. Amer. Math. Soc. (N.S.) \textbf{16} (1987), no. 2, 282--286.
%\bibitem{AdSp88}A. Adolphson and S. Sperber, \emph{On the degree of the L-function associated with an exponential
%sum}, Compositio Math. \textbf{68} (1988), no.2, 125-159.

%\bibitem{Ba15}I. Baoulina, \emph{On the solvability of certain equations over finite fields}, Topics in finite fields, 19--27, Contemp. Math., 632, Amer. Math. Soc., Providence, RI, 2015.

%\bibitem{BGHM} B.~Bank, M.~Giusti, J.~Heintz and G.~Mbakop, \emph{Polar varieties and efficient real equation solving: The hypersurface case}, J. Complexity \textbf{13}  (1997), no. 1,  5--27.

\bibitem{CaMa06}
	A.~Cafure and G.~Matera, \emph{Improved explicit estimates on the
		number of
		solutions of equations over a finite field}, Finite Fields Appl. \textbf{12}
	(2006), no.~2, 155--185.

\bibitem{CaMa07}
	A. Cafure and G. Matera, \emph{An effective {Bertini} theorem and the number of
		rational points
		of a normal complete intersection over a finite field}, Acta Arith.
	\textbf{130} (2007), no.~1, 19--35.
%	
%	\bibitem{CaMaPr12}
%	A.~Cafure, G.~Matera, and M.~Privitelli, \emph{Singularities of
%	symmetric hypersurfaces and {Reed}-{Solomon} codes}, Adv. Math. Commun. \textbf{6} (2012), no.~1, 69--94.
%	
%	\bibitem{CaMaPr15}  
%	A. Cafure, G. Matera and M. Privitelli,
%	 \emph{Polar varieties, {Bertini's} theorems and number of
%	points of
%		singular complete intersections over a finite field}, Finite Fields Appl.
%	\textbf{31} (2015), 42--83.
%	
	



	
	
	

	
\bibitem{Cao2016} X. Cao, W-S. Chou and J. Gu, \emph{On the number of solutions of certain diagonal equations over finite fields}, Finite Fields Appl. \textbf{42} (2016), 225--252.	
	
	
	\bibitem{CaGaHe91}
	L.~Caniglia, A.~Galligo, and J.~Heintz, \emph{Equations for the
		projective
		closure and effective {Nullstellensatz}}, Discrete Appl. Math. \textbf{33}
	(1991), 11--23.
\bibitem{Ca53} L. Carlitz, \emph{Some special equations in a finite field}, Pacific J. Math. \textbf{3} (1953), 13--24. 
	
%\bibitem{Ca56} L. Carlitz, \emph{Solvability of certain equations in a finite field}, Quart. J. Math. Oxford Ser. (2) 7 (1956), 3–4.
	%\bibitem{CaLeMiSt61}L. Carlitz, D. J. Lewis, W. H. Mills and E. G. Straus, \emph{Polynomials over Finite Fields with Minimal Value Sets}, Mathematika, \textbf{8} (1961), 121-130.

%\bibitem{CaRuVe08} N. Castro and Ivelisse Rubio and Jos{\'e} M. Vega, \emph{Divisibility of exponential sums and solvability of certain equations over finite fields}, Q. J. Math. \textbf{60} (2009), no. 2, 169--181.
\bibitem{CRGF08} F. N. Castro, I. Rubio, P. Guan and R. Figueroa, \emph{On systems of linear and diagonal equation of degree $p^i+1$ over finite fields of characteristic $p$}, Finite Fields Appl. {\bf 14} (2008), no.~3, 648--657

%\bibitem{CeMaPePr14} E. Cesaratto, G. Matera, M. Pérez, and M.Privitelli, \emph{On the value set of small families of polynomials over a finite field}. I. J. Combin. Theory Ser. A \textbf{124} (2014), 203--227.
	
	%\bibitem{CeMaPe17}
%	E.~Cesaratto, G.~Matera, and M.~{P\'erez}, \emph{The distribution of
%		factorization patterns on linear families of polynomials over a finite
%		field}, Combinatorica \textbf{37} (2017), no.~5, 805--836.
%		
%		\bibitem{ChMu07}Q. Cheng and E. Murray, \emph{On deciding deep holes of Reed-Solomon codes. Theory and applications of models of computation}, 296--305, Lecture Notes in Comput. Sci., 4484, Springer, Berlin, 2007.
%		
		\bibitem{ChMuWa08} W.-S. Chou, G. L. Mullen and B. Wassermann, \emph{On the number of solutions of equations of
Dickson polynomials over finite fields}, Taiwanese J. Math.\textbf{12} (2008), 917–931.

	
%	\bibitem{ChDrMa92}
%	Z.~{Chatzidakis}, L.~{van den Dries}, and A.~{Macintyre},
%	\emph{Definable sets
%		over finite fields}, J. Reine Angew. Math. \textbf{427} (1992), 107--135.
%	
%	\bibitem{Cohen70}
%	S.~Cohen, \emph{The distribution of polynomials over finite fields},
%	Acta
%	Arith. \textbf{17} (1970), 255--271.
%	
%	\bibitem{Cohen72}
%	\bysame, \emph{Uniform distribution of polynomials over finite
%		fields}, J.
%	Lond. Math. Soc. (2) \textbf{6} (1972), no.~1, 93--102.
%	
%	
	\bibitem{CLO92} D. Cox, J. Little, and D. O'Shea, \emph{Ideals, Varieties, and Algorithms: an introduction to computational algebraic geometry and commutative algebra}. Undergrad. Texts Math. Springer, New York, 1992.
	
%	\bibitem[DKS13]{DaKrSz13}
%	C.~{D'Andrea}, T.~Krick, and A.~Szanto, \emph{Subresultants in
%		multiple roots},
%	Linear Algebra Appl. \textbf{438} (2013), no.~5, 1969--?1989.
	
%	\bibitem{Danilov94}
%	V.~Danilov, \emph{Algebraic varieties and schemes}, Algebraic
%	Geometry I (I.R.
%	Shafarevich, ed.), Encyclopaedia of Mathematical Sciences, vol.~23, Springer,
%	Berlin Heidelberg New York, 1994, pp.~167--307.
%	
	\bibitem{De74} P. Deligne, \emph{La conjecture de Weil. I}, Inst. Hautes Etudes Sci. Publ. Math. (1974), no. 43, 273--307.
	
	
	%\bibitem{DM02} S. De Marchi, 
	%\emph{Polynomials arising in factoring generalized Vandermonde determinants II: A condition for monicity},
	%Appl. Math. Lett. \textbf{15} (2002), no. 5, 627--632. 
	
	\bibitem{Eisenbud95}
	D.~Eisenbud, \emph{Commutative algebra with a view toward algebraic
		geometry},
	Grad. Texts in Math., vol. 150, Springer, New York, 1995.
	
%	\bibitem{Ernst00}
%	T.~Ernst, \emph{Generalized {Vandermonde} determinants}, Report
%	2000:6
%	Matematiska Institutionen, Uppsala Universitet, available at {\tt
%		http://www.math.uu.se/research/pub/}, 2000.
	
	%\bibitem{Fe06}B. Felszeghy, \emph{ On the solvability of some special equations over finite fields}, Publ. Math. Debrecen \textbf{68} (2006), no. 1-2, 15–23.
	
%	\bibitem{FlFuGoPaPo06}
%	P.~Flajolet, E.~Fusy, X.~Gourdon, D.~Panario, and N.~Pouyanne,
%	\emph{A hybrid
%		of {Darboux}'s method and singularity analysis in combinatorial asymptotics},
%	Electron. J. Comb. \textbf{13} (2006), no.~1, research paper r103.
%	
%	\bibitem{FlGoPa01}
%	P.~Flajolet, X.~Gourdon, and D.~Panario, \emph{The complete analysis
%		of a
%		polynomial factorization algorithm over finite fields}, J. Algorithms
%	\textbf{40} (2001), no.~1, 37--81.
%	
%	\bibitem{FlSe09}
%	P.~Flajolet and R.~Sedgewick, \emph{Analytic combinatorics},
%	Cambridge Univ.
%	Press, Cambridge, 2009.
%	
%	\bibitem{FrHaJa94}
%	M.~Fried, D.~Haran, and M.~Jarden, \emph{Effective counting of the
%		points of
%		definable sets over finite fields}, Israel J. Math. \textbf{85} (1994),
%	no.~1-3, 103--133.
%	
%	\bibitem{FrSm84}
%	M.~Fried and J.~Smith, \emph{Irreducible discriminant components of
%		coefficient
%		spaces}, Acta Arith. \textbf{44} (1984), no.~1, 59--72.
%	
	\bibitem{Fulton84}
	W.~Fulton, \emph{Intersection theory}, Springer, Berlin Heidelberg
	New York,
	1984.
	
%	\bibitem{GaHoPa99}
%	S.~{Gao}, J.~{Howell}, and D.~{Panario}, \emph{Irreducible
%		polynomials of given
%		forms}, Finite fields: theory, applications, and algorithms. Fourth
%	international conference, Waterloo, Ontario, Canada, August 12--15, 1997,
%	Amer. Math. Soc., Providence, RI, 1999, pp.~43--54.
%	
	\bibitem{GaGe99}
J.~von~zur {Gathen} and J.~Gerhard, \emph{Modern computer algebra},
Cambridge
	Univ. Press, Cambridge, 1999.
%	
%	\bibitem{GaMa18}
%	J.~von~zur {Gathen} and G.~Matera, \emph{Explicit estimates for
%		polynomial
%		systems defining irreducible smooth complete intersections}, Preprint {\tt
%		arXiv:1512.05598 [math.NT]}, 2018, to appear in Acta Arith.
%	
%	\bibitem{GeCzLa92}
%	K.~Geddes, S.~Czapor, and G.~Labahn, \emph{Algorithms for computer
%		algebra},
%	Kluwer Acad. Publ., Dordrecht, 1992.
%	
	\bibitem{GhLa02a}
	S.~Ghorpade and G.~Lachaud, \emph{{\'Etale} cohomology, {Lefschetz}
		theorems
		and number of points of singular varieties over finite fields}, Mosc. Math.
	J. \textbf{2} (2002), no.~3, 589--631.
	
	\bibitem{GhLa02}
	S.~Ghorpade and G.~Lachaud, \emph{Number of solutions of equations over finite fields
		and a
		conjecture of {Lang} and {Weil}}, Number Theory and Discrete Mathematics
	(Chandigarh, 2000) (New Delhi) (A.K.~Agarwal et~al., ed.), Hindustan Book
	Agency, (2002), 269--291.
	
%	\bibitem{Gibson98}
%	C.~Gibson, \emph{Elementary geometry of algebraic curves: an
%		undergraduate
%		introduction}, Cambridge Univ. Press, Cambridge, 1998.
%	
%	\bibitem{GoGa98}
%	S.~Golomb and P.~Gaal, \emph{On the number of permutations of {$n$}
%		objects
%		with greatest cycle length {$k$}}, Adv. Appl. Math. \textbf{20} (1998),
%	no.~1, 98--107.
%	
%	\bibitem[GoWi2010]{GoWi10} D. Gomez and A. Winterhof, ‘Waring’s problem in finite fields with Dickson polynomials’,
%	Finite fields: Theory and applications, Contemp. Math., vol. 477, Amer. Math. Soc., 2010,
%	185–192.
	
%	\bibitem{GrKnPa94}
%	R.~{Graham}, D.~{Knuth}, and O.~{Patashnik}, \emph{{Concrete
%			mathematics: a
%			foundation for computer science}}, 2nd ed., Addison--Wesley, Reading,
%	Massachusetts, 1994.
%	
%	\bibitem{GrKn90}
%	D.~{Greene} and D.~{Knuth}, \emph{{Mathematics for the analysis of
%			algorithms}}, 3rd ed., Birkh\"auser, Boston Basel Berlin, 1990.
%	
%	\bibitem{Ha16}
%	J.~Ha, \emph{Irreducible polynomials with several prescribed
%		coefficients},
%	Finite Fields Appl. \textbf{40} (2016), 10--25.
%	
	\bibitem{Harris92}
	J.~Harris, \emph{Algebraic geometry: a first course}, Grad. Texts in
	Math.,
	vol. 133, Springer, New York Berlin Heidelberg, 1992.
	
	\bibitem{He85} T. Helleseth, On the covering radius of cyclic linear codes and arithmetic codes, Discrete Appl. Math. {\bf 11} (1985), no.~2, 157--173.
	
	\bibitem{Heintz83}
	J.~Heintz, \emph{{Definability} and fast quantifier elimination in
		algebraically closed fields}, Theoret. Comput. Sci. \textbf{24} (1983),
	no.~3, 239--277.
	
\bibitem{JGC20}	K. Jiang, W. Gao, W. Cao,
\emph{Counting solutions to generalized Markoff-Hurwitz-type equations in finite fields},
Finite Fields Appl. \textbf{62} (2020).
	
	
	%\bibitem{Katz01} N. Katz, \emph{Sums of Betti numbers in arbitrary characteristic}, Finite Fields Appl. \textbf{7} (2001),
	%29--44.
	
%	\bibitem{KnKn90b}
%	A.~Knopfmacher and J.~Knopfmacher, \emph{The distribution of values
%		of
%		polynomials over a finite field}, Linear Algebra Appl. \textbf{134} (1990),
%	145--151.
%	
%	\bibitem{Knuth98}
%	D.E. Knuth, \emph{The art of computer programming {II}:
%		{Semi--numerical}
%		algorithms}, 3rd. ed., vol.~2, Addison-Wesley, Reading, Massachusetts, 1998.
%	
	\bibitem{Kunz85}
	E.~Kunz, \emph{Introduction to commutative algebra and algebraic
		geometry},
	Birkh{\"a}user, Boston, 1985.
	
%	\bibitem{LaPr02}
%	A.~Lascoux and P.~Pragracz, \emph{Jacobian of symmetric
%		polynomials}, Ann.
%	Comb. \textbf{6} (2002), no.~2, 169--172.
%	
	\bibitem{LaRo15}
	G.~Lachaud and R.~Rolland, \emph{On the number of points of
		algebraic sets over
		finite fields}, J. Pure Appl. Algebra \textbf{219} (2015), no.~11,
	5117--5136.
	
	\bibitem{LiNi83}
	R.~Lidl and H.~Niederreiter, \emph{Finite fields}, Addison--Wesley,
	Reading,
	Massachusetts, 1983.
	
%	\bibitem{MaPePr14}
%	G.~Matera, M.~{P\'erez}, and M.~Privitelli, \emph{On the value set
%		of small
%		families of polynomials over a finite field, {II}}, Acta Arith. \textbf{165}
%	(2014), no.~2, 141--179.
%	
%	\bibitem{MaPePr16}
%	G.~Matera, M.~{P\'erez}, and M.~Privitelli, \emph{Explicit estimates for the number of rational points
%		of singular
%		complete intersections over a finite field}, J. Number Theory \textbf{158}
%	(2016), no.~2, 54--72.
%	\bibitem{MaPePr19} G. Matera, M.Pérez and M. Privitelli, \emph{ Factorization patterns on nonlinear families of univariate polynomials over a finite field}, J Algebr. Comb. (2019), 1-51.
%	
	
	%\bibitem{Matsamura}H. Matsumura, Commutative algebra, Benjamin, 1980.
	
%	\bibitem{Merca13}
%	M.~Merca, \emph{A note on the determinant of a
%		{Toeplitz--Hessenberg} matrix},
%	Spec. Matrices \textbf{1} (2013), 10--16.
	
\bibitem{Mo63} L. J. Mordell, \emph{On a special polynomial congruence and exponential sums}, 1963 Calcutta Math. Soc. Golden Jubilee Commemoration Vol, 29--32 Calcutta Math. Soc., Calcutta.

\bibitem{MoCa03} O. Moreno\ and\ F. N. Castro, Divisibility properties for covering radius of certain cyclic codes, IEEE Trans. Inform. Theory {\bf 49} (2003), no.~12, 3299--3303.	
	
%	\bibitem{Muir60}
%	T.~Muir, \emph{The theory of determinants in the historical order of
%		development}, Dover Publications Inc., New York, 1960.
\bibitem{MuPa13}	Gary L. Mullen and Daniel Panario, Handbook of Finite Fields (1st ed.), Chapman and Hall/CRC, 2013.
	
%\bibitem[OsSh2011]{OsSh2011} Ostafe, Alina ; Shparlinski, Igor E. / On the waring problem with Dickson polynomials in finite fields. In: Proceedings of the American Mathematical Society. 2011 ; Vol. 139, No. 11. pp. 3815-3820.
%	
%	\bibitem{PaSa04}
%	L.M. Pardo and J.~{San Mart\'{\i}n}, \emph{Deformation techniques to
%		solve
%		generalized {Pham} systems}, Theoret. Comput. Sci. \textbf{315} (2004),
%	no.~2--3, 593--625.
%	
%	\bibitem{Perez16}
%	M.~P{\'e}rez, \emph{An\'alisis probabil\'istico de algoritmos y
%		problemas
%		combinatorios sobre cuerpos finitos}, Ph.D. thesis, Univ. Buenos Aires,
%	Argentina, 2016.

\bibitem{PP20} M.~P{\'e}rez and M. Privitelli, \emph{Estimates on the number of rational solutions of variants of diagonal equations over finite fields}, Finite Fields and Appl. \textbf{68} (2020). In press.
%	
%	\bibitem{Pollack13}
%	P.~Pollack, \emph{Irreducible polynomials with several prescribed
%		coefficients}, Finite Fields Appl. \textbf{22} (2013), 70--78.
%		
%\bibitem{Ro09}  F. Rodier, \emph{Borne sur le degré des polynômes presque parfaitement non-linéaires}, Arithmetic, geometry, cryptography and coding theory, 169--181, Contemp. Math., 487, Amer. Math. Soc., Providence, RI, 2009.		
		
	
	\bibitem{Shafarevich94}
	I.R. Shafarevich, \emph{Basic algebraic geometry: {Varieties} in
		projective
		space}, Springer, Berlin Heidelberg New York, 1994.
		
		
		\bibitem{Spackman79} K. W. Spackman, \emph{Simultaneous solutions to diagonal equations over finite fields}, J. Number Theory {\bf 11} (1979), no.~1, 100--115.
	\bibitem{Spackman81} K. W. Spackman, \emph{On the number and distribution of simultaneous solutions to diagonal congruences}, Canadian J. Math. {\bf 33} (1981), no.~2, 421--436.
%	L.~Shepp and S.~Lloyd, \emph{Ordered cycle lengths in a random
%		permutation},
%	Trans. Amer. Math. Soc. \textbf{121} (1996), 340--357.
%	
%	\bibitem{Shoup05}
%	V.~Shoup, \emph{A computational introduction to number theory and
%		algebra},
%	Cambridge Univ. Press, Cambridge, 2005.
\bibitem{T65 A} A. Tiet\"{a}v\"{a}inen, \emph{On the non-trivial solvability of some equations and systems of equations in finite fields}, Ann. Acad. Sci. Fenn. Ser. A I no. {\bf 360} (1965), 38 pp.
\bibitem{T65} A. Tiet\"{a}v\"{a}inen, \emph{On systems of linear and quadratic equations in finite fields}, Ann. Acad. Sci. Fenn. Ser. A I no. {\bf 382} (1965), 5 pp.
	\bibitem{T64} A. Tiet\"{a}v\"{a}inen,\emph{ On the non-trivial solvability of some systems of equations in finite fields}, Ann. Univ. Turku. Ser. A I {\bf 71} (1964), 5 pp.
	\bibitem{T67} A. Tiet\"{a}v\"{a}inen, \emph{On the solvability of equations in incomplete finite fields}, Ann. Univ. Turku. Ser. A I {\bf 102} (1967), 13 pp.
	\bibitem{Vogel84}
	W.~Vogel, \emph{Results on {B\'ezout}'s theorem}, Tata Inst. Fundam.
	Res. Lect.
	Math., vol.~74, Tata Inst. Fund. Res., Bombay, 1984.
\bibitem{Weil49} A. Weil, \emph{Numbers of solutions of equations in finite fields}, Bull. Amer. Math. Soc. {\bf 55} (1949), 497--508.
\bibitem{W98} J. Wolfmann, \emph{Some systems of diagonal equations over finite fields}, Finite Fields Appl. {\bf 4} (1998), no.~1, 29--37.
%	
%	\bibitem{Zassenhaus69}
%	H.~Zassenhaus, \emph{On {Hensel} factorization {I}}, J. Number
%	Theory
%	\textbf{1} (1969), 291--311.
\bibitem{Zh10} X. Zeng L. Hu, W. Jiang, Q. Yue and X. Cao, \emph{The weight distribution of a class of $p$-ary cyclic codes}, Finite Fields Appl. {\bf 16} (2010), no.~1, 56--73.
%
\bibitem{Zh15} D. Zheng,  X. Wang, X. Zeng and  L. Hu, \emph{The weight distribution of a family of $p$-ary cyclic codes}, Des. Codes Cryptogr. {\bf 75} (2015), no.~2, 263--275.

 






\end{thebibliography}
\end{document}